\documentclass[numbers,webpdf,imaiai]{ima-authoring-template}

\usepackage{braket,amsfonts}
\usepackage{amsmath,amssymb}
\usepackage{booktabs}
\usepackage{siunitx}
\usepackage{array}
\usepackage[caption=true]{subfig}
\usepackage{caption}
\usepackage{soul}
\usepackage{pgfplots}
\pgfplotsset{compat=1.18}
\usepackage{hyperref}

\usepackage{graphicx,epstopdf}
\usepackage{amsopn}

\usepackage{xspace}
\usepackage{bold-extra}

\usepackage{etoolbox}
\usepackage{nerdnote}
\usepackage[latte,whitebg]{nerdcolor}
\AtBeginEnvironment{quotation}{\color{blue}}

\def\norm#1{\|#1\|}
\DeclareMathOperator{\fl}{\operatorname{f\kern.2ptl}} 

\usepackage{standalone}
\usepackage{tikz}
\usetikzlibrary{shapes.geometric, arrows, positioning}

\graphicspath{{Fig/}}

\theoremstyle{thmstyletwo}%
\newtheorem{theorem}{Theorem}
\newtheorem{remark}{Remark}%

\newtheorem{lemma}{Lemma}

\numberwithin{equation}{section}

\begin{document}

\DOI{DOI HERE}
\copyrightyear{2021}
\vol{00}
\pubyear{2021}
\access{Advance Access Publication Date: Day Month Year}
\appnotes{Paper}
\copyrightstatement{Published by Oxford University Press on behalf of the Institute of Mathematics and its Applications. All rights reserved.}
\firstpage{1}


\title[Mixed Precision General ADI Method]{Mixed Precision General
Alternating-Direction Implicit Method for Solving Large Sparse Linear Systems}

\author{Jifeng Ge
\address{School of Mathematics and Computational Science, Xiangtan University, Xiangtan, Hunan, China}}
\author{Bastien Vieublé
\address{Academy of Mathematics and Systems Science, CAS, 
100190 Beijing, China}}
\author{Juan Zhang*
\address{Key Laboratory of Intelligent Computing and 
Information Processing of Ministry of Education, Hunan Key Laboratory for
Computation and Simulation in Science and Engineering, Hunan, China}}

\authormark{Jifeng Ge, Bastien Vieublé and Juan Zhang}

\corresp[*]{Corresponding author: \href{zhangjuan@xtu.edu.cn}{zhangjuan@xtu.edu.cn}}



\abstract{
In this article, we introduce a three-precision formulation of the General 
Alternating-Direction Implicit method (GADI) designed to 
accelerate the solution of large-scale sparse linear systems $Ax=b$. GADI is a
framework that can represent many existing Alternating-Direction 
Implicit (ADI) methods. These methods are a class of linear solvers based on a 
splitting of $A$ such that the solution of the original linear system can be 
decomposed into the 
successive computation of easy-to-solve structured subsystems. Our 
proposed mixed precision scheme for GADI solves these subsystems in low 
precision to reduce the overall execution time while computing the residual and
solution update in high precision to enable the solution to converge to high 
accuracy. We develop a rounding error analysis of mixed precision GADI that 
establishes the rates of convergence of the forward and backward
errors to certain limiting accuracies. 
Our analysis also highlights the conditions on the
splitting matrices under which mixed precision GADI is guaranteed to converge 
for a given set of precisions.
We then discuss a systematic and robust strategy for selecting the
GADI regularization parameter $\alpha$, whose adjustment is critical for
performance.
Specifically, our proposed strategy makes use of a Gaussian Process Regression (GPR) model
trained on a dataset of low-dimensional problems to initialize $\alpha$. 
Finally, we proceed to a performance analysis of mixed precision GADI on an NVIDIA A100
GPU to validate our approach. Using low precision (Bfloat16 or FP32) to solve the subsystems, we obtain speedups of $2.6\times$, $1.7\times$, 
and $3.1\times$ over a full double precision GADI
implementation on large-scale 2D, 3D convection-diffusion and complex 
reaction-diffusion problems (up to $1.3\times 10^{8}$ unknowns), respectively.
}

\maketitle

~


\section{Introduction}
\label{sec:intro}

Solving large-scale nonsingular sparse linear systems of the form 
\begin{equation}
    Ax = b, \quad \text{where} \quad  A \in \mathbb{R}^{n \times n} \quad
    \text{and} \quad 0\ne b \in \mathbb{R}^{n},
    \label{eq:linear_equation}
\end{equation}
is a cornerstone 
of computational science and engineering. As problem sizes grow, the demand for
faster and more efficient solvers has become critical. The Alternating-Direction
Implicit (ADI) method, pioneered by Douglas and Rachford~\cite{douglas1956numerical} and
Douglas~\cite{douglas1955numerical}, is a popular method for solving the large matrix
equations that arise in systems theory and 
control~\cite{doi:10.1137/130912839}, and can be formulated to construct
solutions in a memory-efficient way~\cite{doi:10.1137/S0895479801384937,BENNER20091035}. 
To solve $Ax=b$ efficiently, ADI methods look for splitting the matrix \( A \) into 
\( M + N \), where \( M \) and \( N \) are chosen to decompose the original 
linear system into the computation of successive, simpler, subsystems.
By using a general splitting formulation of matrices and an extra 
splitting parameter, the General Alternating-Direction Implicit (GADI) 
framework, introduced by Jiang et al.~\cite{doi:10.1137/21M1450197}, can be
used to express many existing and new ADI
methods~\cite{2007On, bai2003hermitian, WANG20132352}.
The GADI iterations make use of two parameters that adjust the numerical stability
and rate of convergence: a regularization parameter 
\( \alpha > 0 \) and an extrapolation parameter \( \omega \in [0, 2) \). A key strength
of GADI is its guaranteed convergence under broad conditions, making it a 
reliable foundation for developing new ADI methods.

A  recent parallel advancement in high-performance computing is the rise of 
mixed precision algorithms driven by the increasing support of 
resource-efficient low precisions in hardware. We list the floating-point 
arithmetics we use or mention in this article in Table~\ref{tb1}. This creates
a compelling opportunity: by strategically combining different precisions 
within an algorithm, one can achieve significant performance gains while
maintaining high accuracy on the final solution. Such mixed precision 
algorithms have been widely employed for the solution of linear systems: for
direct solvers~\cite{doi:10.1137/17M1140819,abhl23}, Krylov-based iterative
solvers~\cite{lld22}, or multigrid solvers~\cite{mbt21}. We 
refer the reader to the surveys by Higham and Mary~\cite{hima22} or Abdelfattah
et al.~\cite{aabc21} for more insight into the employment of mixed precision 
within numerical linear algebra. Yet, except for the recent work~\cite{scsa25},
the literature has not extensively investigated how ADI methods can benefit from mixed 
precision. 

\begin{table}[htbp]
    \centering
    \caption{Parameters for floating-point arithmetics: number of bits for
    the significand (including the implicit leading bit), number of bits for 
    the exponent, unit roundoff, and range.}
	\begin{tabular}{llllll}
	    \toprule
	    Arithmetic & Significand & Exponent & Unit roundoff  & Range                \\
	    \midrule
	    Bfloat16   & 8           & 8        & $3.91\times 10^{-3}$ & $10^{\pm 38}$ \\
	    FP16       & 11           & 5        & $4.88 \times 10^{-4}$ & $10^{\pm 5}$ \\
	    FP32       & 24          & 8        & $5.96\times 10^{-8}$ & $10^{\pm 38}$  \\
	    FP64       & 53          & 11       & $1.11\times 10^{-16}$ & $10^{\pm 308}$ \\
	    FP128      & 113          & 15       & $9.63\times 10^{-35}$ & $10^{\pm 4932}$ \\
	    \bottomrule
	\end{tabular}
    \label{tb1}
\end{table}

In this context, our aim is to introduce a new mixed precision scheme for the GADI framework, which is applicable and efficient across a wide range of ADI methods, and that may lead to substantial resource savings for the solution of different matrix equations and PDEs. We summarize our contributions as follows:
\begin{itemize}
    \item \textbf{The GADI framework in three precisions.} In 
        section~\ref{sec:mix-gadi}, we propose a mixed precision scheme for the GADI framework
        that integrates three different precision parameters. The computation 
        of the solutions to the subsystems, which is generally the most 
        resource-intensive part of the GADI iterations, is performed in low 
        precision. The residual and solution update, which are less 
        computationally intensive, are performed in working precision, which is
        the precision at which the solution is delivered to the user. The 
        residual can alternatively be computed in an even more accurate extra
        precision to further improve the solution accuracy, leading to up to
        three different precisions within GADI.
    \item \textbf{Rounding error analysis.} In section~\ref{sec:ea}, we proceed
        to the rounding error analysis of mixed precision GADI, from which we
        derive conditions for the convergence of the computed solution iterates and 
        limiting accuracies of the
        forward and backward errors. In 
        addition to revealing the roles of the three different precisions on 
        the numerical behavior of GADI, our analysis improves on previous ADI
        rounding error analyses by, first, achieving sharper limiting accuracy 
        bounds and, second, providing expressions of the convergence rates 
        under rounding errors.
    \item \textbf{Systematic regularization parameter selection.} The
        convergence and efficiency of the GADI framework are highly sensitive 
        to the choice of the parameter \( \alpha \). This dependency becomes 
        even more critical in a mixed precision setting, where the limited 
        range and accuracy of formats like FP16 can amplify numerical errors 
        and lead to instability if parameters are not chosen carefully. To 
        balance computational efficiency and numerical stability, we propose a
        systematic strategy to select $\alpha$ in section~\ref{sec:pr}. This 
        strategy uses a Gaussian Process Regression
        (GPR)~\cite{NIPS1995_7cce53cf}, which is a
        machine learning technique that we use to build a predictive optimal
        $\alpha$-parameter model based on data from smaller, cheaper-to-run 
        linear systems. 
    \item \textbf{Performance analysis.} Finally, in section~\ref{sec:ne}, we 
        apply mixed precision GADI to several large-scale problems, 
        including 2D and 3D convection-diffusion equations and a complex 
        reaction-diffusion equation. We benchmark mixed precision GADI on 
        systems up to $1.3\times 10^{8}$ unknowns against
        full double precision GADI, NVIDIA’s cuDSS solver, and
        mixed precision GMRES-based iterative refinement.
\end{itemize}

\section{Mixed precision GADI}
\label{sec:mix-gadi}

\subsection{Background on the GADI framework}
\label{subsec:gadi_framework}
The GADI framework~\cite{doi:10.1137/21M1450197} is an iterative method
for solving the linear system~\eqref{eq:linear_equation}. It is based on a 
splitting of the matrix \(A\) into \(A = M + N\), where \(M\) and \(N\) are 
chosen to make subsequent linear systems easier to solve. Given a 
regularization parameter \(\alpha > 0\) and an extrapolation parameter \(\omega
\in [0, 2)\), the GADI iteration proceeds in two steps
\begin{enumerate}
    \item Solve \( (\alpha I + M) x^{(k + \frac{1}{2})} = (\alpha I - N) x^{(k)} + b \).
    \item Solve \( (\alpha I + N) x^{(k + 1)} = (N - (1 - \omega) \alpha I) x^{(k)} + (2 - \omega) \alpha x^{(k + \frac{1}{2})} \). 
    \label{gadi_step}
\end{enumerate}
This can be expressed as \( x^{(k + 1)} = T_F(\alpha, \omega) x^{(k)} + G(\alpha,
\omega) \), where $T_F(\alpha, \omega) = (\alpha I + N)^{-1} (\alpha I + M)^{-1}
(\alpha^2 I + MN - (1 - \omega) \alpha A)$ and $G(\alpha, \omega) = 
(2 - \omega)\alpha(\alpha I + N)^{-1}(\alpha I + M)^{-1}b$. Hence, the
asymptotic convergence of 
the GADI iteration is determined by the spectral radius of the iteration matrix
$T_F(\alpha, \omega)$, denoted as $\rho(T_F(\alpha, \omega))$, which is the maximum
of the absolute values of its eigenvalues. 
We also denote $T_B(\alpha, \omega) = A T_F(\alpha, \omega)A^{-1}$, the similarity
transform of the iteration matrix which satisfies $\rho(T_B(\alpha, \omega)) =
\rho(T_F(\alpha, \omega))$.
Convergence is guaranteed under 
broad conditions, as stated in the following lemma.
\begin{lemma}[\cite{doi:10.1137/21M1450197} Theorem~2.2]
    The GADI framework converges to the unique solution \( x \) of
    \( Ax = b \) for any \( \alpha > 0 \) and \( \omega \in [0, 2) \), provided
    \( M \) is positive definite and \( N \) is either positive (semi-)definite
    or skew-Hermitian. Moreover, the spectral radius of the iteration matrix 
    satisfies
    \begin{align*}
    \rho(T_F(\alpha, \omega)) = \rho(T_B(\alpha, \omega)) < 1.
    \end{align*}\label{lem:gadi_convergence}
\end{lemma}

Many classical ADI schemes fall under the GADI framework through appropriate
choices of \(M\), \(N\), and \(\omega\). Examples include:
\begin{itemize}
    \item Peaceman-Rachford splitting (equivalently, HSS) with 
        Hermitian/skew-Hermitian splitting \(A=H+S\) for non-Hermitian positive
        definite linear systems~\cite{peaceman1955numerical,bai2003hermitian}.
    \item Douglas-Rachford splitting (DRS), widely used for parabolic PDEs such
        as heat conduction and related time-stepping 
        splits~\cite{douglas1956numerical}.
    \item Normal/skew-Hermitian splitting (NSS) and 
        positive-definite/skew-Hermitian splitting (PSS), targeting classes of 
        saddle-point and non-Hermitian systems~\cite{benzi2003,2007On}.
    \item Generalized HSS variants and ADI methods for matrix equations (e.g., 
        Sylvester/Lyapunov), prominent in systems and 
        control~\cite{BENNER20091035,bai2011hermitian,WANG20132352}.
\end{itemize}

\subsection{Similarities between GADI and Iterative Refinement}
\label{subsec:mp_ir}
Iterative Refinement (IR) is one of the most popular, 
efficient, and robust methods for leveraging low precision arithmetics to 
solve~\eqref{eq:linear_equation} while still providing a high accuracy 
solution. The mixed precision generalized IR scheme proposed 
in~\cite{doi:10.1137/17M1140819} repeats the following three steps until 
convergence:
\begin{enumerate}
    \item Compute the residual $r^{(k)} = b - Ax^{(k)}$ in extra precision
        $u_{r}$ (e.g., FP64 or FP128 for a FP64 solution accuracy).
    \item Compute the correction $y^{(k)}$ by solving the linear system
        $Ay^{(k)} = r^{(k)}$ with a computationally effective but inaccurate 
        linear solver in precision $u_{s}$. Here, we loosely refer to $u_s$
        as the ``precision'' of the solver for simplicity of exposition. 
        However, we emphasize that it is instead a more general measure of the
        quality of the solutions computed by the inaccurate solver, which can 
        itself be in mixed precision or leverage other forms of numerical 
        approximations. \label{step:ir-correq}
    \item Compute the solution update \(x^{(k+1)} = x^{(k)} + y^{(k)}\) in
        working precision $u$ (e.g., FP64 for a FP64 solution accuracy).
\end{enumerate}
Various IR implementations were shown to be highly successful at significantly
reducing the computational resources required for dense~\cite{htdh18} and 
sparse~\cite{abhl23,lld22} linear systems arising from real-life applications 
and running on modern HPC architectures.

While IR is not the direct topic of this article, the observation that it 
shares algorithmic similarities with GADI motivates the mixed precision
approach we develop. Indeed, the GADI iteration of 
section~\ref{subsec:gadi_framework} can be rewritten in a four-step form as
outlined by Algorithm~\ref{alg1}. Hence, similarly to IR, GADI computes at
each iteration $k$ the residual $r^{(k)}$ (line~\ref{line:alg1-res} of 
Algorithm~\ref{alg1}) from which it extracts a correction $y^{(k)}$ 
(lines~\ref{line:alg1-sys1} and~\ref{line:alg1-sys2} of Algorithm~\ref{alg1}).
It then updates the solution using this correction to obtain the next iterate
$x^{(k+1)}$ (line~\ref{line:alg1-update} of Algorithm~\ref{alg1}). These 
similarities are actually not surprising and originate from the fact that both
procedures are stationary iterative 
methods~\cite[Section~5.5.1]{LiesenStrakos2013}.

\subsection{A mixed precision scheme for GADI}

Our intention is therefore to adapt the well-known and successful mixed 
precision IR scheme to GADI. The resulting mixed precision GADI algorithm is
represented by Algorithm~\ref{alg1}. Specifically, we use an extra precision $u_r$ to compute
the residual at line~\ref{line:alg1-res}, we use the working precision $u$ to 
compute the update at line~\ref{line:alg1-update}, and we use a low precision 
$u_s$ to solve the subsystems at lines~\ref{line:alg1-sys1} 
and~\ref{line:alg1-sys2}, which are (in most cases) the resource intensive part
of the GADI iteration. The rest of this article focuses on establishing the relevance and efficiency of
this proposed mixed precision algorithm. To achieve this, we proceed in two stages. First, in 
section~\ref{sec:ea}, we give theoretical guarantees that mixed precision
GADI still computes high accuracy solutions even while leveraging low
precision. Second, in section~\ref{sec:ne}, we demonstrate that this 
mixed precision approach can effectively reduce resource consumption for the 
solution of real-life large problems.

\begin{algorithm}
    \caption{Mixed Precision GADI}
    \label{alg1}
    \begin{algorithmic}[1]
        \Require: $\alpha,\omega,M,N,k=0,x^{(0)}=0$.
        \Repeat
        \State{Compute the residual $r^{(k)}=b-Ax^{(k)}$.}\Comment{$u_r$}
        \label{line:alg1-res}
        \State{Solve $(\alpha I +
        M)z^{(k)} = r^{(k)}$.}\Comment{$u_s$}
        \label{line:alg1-sys1}
        \State{Solve $(\alpha I +
        N)y^{(k)} = (2-\omega)\alpha z^{(k)}$.}\Comment{$u_s$}
        \label{line:alg1-sys2}
        \State{Compute the next iterate
        $x^{(k+1)}=x^{(k)}+y^{(k)}$.}\Comment{$u_{\phantom{s}}$}
        \label{line:alg1-update}
        \State $k = k+1$.
  \Until{convergence}
    \end{algorithmic}
\end{algorithm}

While similar to an extent, it is worth noting that GADI and IR are also 
intrinsically different. In the case of GADI, $y^{(k)}$ is not computed as an 
(inaccurate) solution of $Ay^{(k)} = r^{(k)}$. Namely, 
lines~\ref{line:alg1-sys1} and~\ref{line:alg1-sys2} of Algorithm~\ref{alg1} are
not solving the system $Ay^{(k)} = r^{(k)}$ as for step~\ref{step:ir-correq} 
of IR in section~\ref{subsec:mp_ir}. It means, in particular, that we cannot simply exploit theoretical 
results of IR developed, for instance, in~\cite{doi:10.1137/17M1140819}, to
carry out our rounding error analysis of mixed precision GADI.

\section{Rounding error analysis of mixed precision GADI}
\label{sec:ea}

The goal of the rounding error analysis of mixed precision GADI developed in 
this section is quite similar to classic analyses of IR. That is, we aim to
determine the convergence rates and the limiting accuracies of the computed 
solution errors as functions of the precisions $u$, $u_r$, $u_s$, and the
properties of the inputs. Our main conclusions are embodied by Theorem~\ref{thm2} 
and~\ref{thm3}, which concern, respectively, the forward error of the computed
solution defined as 
\begin{equation}\label{eq:429833}
    \frac{\norm{\widehat{x}^{(k)}-x}}{\norm{x}},
\end{equation}
and the normwise backward error defined as \cite[Section~7.1]{doi:10.1137/1.9780898718027}
\begin{equation}\label{eq:231477}
    \min \big\{\, \varepsilon: (A+\Delta A)\widehat{x}^{(k)} = b+\Delta b, \; 
    \|\Delta A\|
    \le \varepsilon \norm{A},
                        \; \|\Delta b\| \le \varepsilon \norm{b} \,\big\}
            =  \frac{\norm{b -
            A\widehat{x}^{(k)}}}{\norm{A}\,\norm{\widehat{x}^{(k)}} + 
            \norm{b}},
\end{equation}
where $\widehat{x}^{(k)}$ refers to the computed solution at the $k$th GADI 
iteration, and $x$ is the exact solution of~\eqref{eq:linear_equation}.

\subsection{Notations and assumptions}
\label{sec:notations}
We now summarize our notations and assumptions. Throughout the analysis, we use
the matrix and vector 2-norm. For a nonsingular matrix $A$, we use the normwise
condition number \(\kappa(A)=\|A\|\cdot\|A^{-1}\|\). Our analysis makes use of 
the three precision parameters $u$, $u_{r}$, and $u_{s}$, which can refer to
both the floating-point arithmetic or its unit roundoff, depending on the
context. Throughout our rounding error analysis, these three precision 
parameters are unspecified and can represent any floating-point arithmetic;
see examples of arithmetics that can be used in Table~\ref{tb1}. We only
require that $u_r\leq u \leq u_s$, which means that $u$ is more accurate than (or
equal to) $u_s$, and $u_r$ is more accurate than (or equal to) $u$.

We adopt the standard model of floating-point 
arithmetic~\cite[Section~2.2]{doi:10.1137/1.9780898718027}. The computed value
of a given expression in floating-point with unit roundoff 
$\nu \in \{u,u_r,u_s\}$ is denoted by $\fl_\nu(\cdot)$, and we put a hat on
variables to indicate that they represent computed quantities. 

The error bounds derived in our analysis depend on some constants related to
the problem dimension $n$. As our analysis is a traditional worst-case
analysis, these constants are known to be pessimistic~\cite{hima19a}. For this reason, we do
not keep track of their precise values and gather them into generic functions
$c(n)$. We ensure that these functions are polynomials in $n$ of low degree.

For convenience, we use the notations $\lesssim$ and $\approx$ when dropping 
negligible second-order terms in the error bounds. Hence, writing $a\lesssim b$
for some positive scalars $a$ and $b$ means 
$a\le (1+\mathcal{O}(u)+\mathcal{O}(u_r)+\mathcal{O}(u_s))b$.

In order to obtain sharper error bounds, we will make use of the term $\mu^{(k)}$
defined as 
\begin{equation}\label{eq:947298}
    \norm{b - A\widehat{x}^{(k)}} = \norm{A(x - \widehat{x}^{(k)})} = \mu^{(k)} \norm{A} \norm{x -
    \widehat{x}^{(k)}},
\end{equation}
and satisfying for all GADI iterations $k\geq 0$
\begin{equation}\label{eq:342080}
    \kappa(A)^{-1} \leq \mu^{(k)} \leq 1.
\end{equation}
This is a key quantity that has been introduced to sharpen the rounding
error analyses of mixed precision IR~\cite{cahi17,doi:10.1137/17M1140819}.
Specifically, in~\cite[section~2.1]{cahi17}, it is explained that we can expect
$\mu^{(k)}$ to be close to the ratio between the backward~\eqref{eq:231477} and
forward~\eqref{eq:429833} errors. This ratio can be as low as
$\kappa(A)^{-1}$ for the first iterations of IR and GADI.

We define the matrices $H\in\mathbb{R}^{n\times n}$ and 
$S\in\mathbb{R}^{n\times n}$ as
\begin{equation}
\begin{aligned}
    A=M+N, \quad H=\alpha I + M, \quad S=\alpha I + N,
\end{aligned}
\label{eq:split}
\end{equation}
where \(\alpha > 0\) is the GADI regularization parameter. In addition, we 
assume that $H$ and $S$ are numerically nonsingular relative to the precision
$u_s$, so that the linear systems with $H$ and $S$ at 
lines~\ref{line:alg1-sys1} and~\ref{line:alg1-sys2} of Algorithm~\ref{alg1} are well-posed:
\begin{equation}\label{eq:293478}
    c(n) \kappa(H) u_s < 1/2 \qquad \mbox{and} \qquad c(n) \kappa(S) u_s <
    1/2.
\end{equation}

Finally, we need some assumptions on the accuracies of the linear solvers used
at lines~\ref{line:alg1-sys1} and~\ref{line:alg1-sys2} of Algorithm~\ref{alg1}.
More specifically, we assume that these solvers compute the linear systems
\[
H\, z^{(k)} = r^{(k)} \quad \text{and} \quad S\,
y^{(k)} = (2-\omega)\alpha \, z^{(k)}
\]
in a backward stable manner, such that the computed $\widehat{z}^{(k)}$ and 
$\widehat{y}^{(k)}$ satisfy, respectively,
\begin{equation}
\begin{gathered}
    (H + F_k^{\scriptscriptstyle H})\widehat{z}^{(k)} = \widehat{r}^{(k)} \quad \text{and} \quad  (S +
    F_k^{\scriptscriptstyle S})\widehat{y}^{(k)} = (2 - \omega)\alpha \widehat{z}^{(k)},\\
    \|F_k^{\scriptscriptstyle H}\|\leq c(n)\,u_s\,\|H\| \quad \text{and} \quad
    \|F_k^{\scriptscriptstyle S}\|\leq c(n)\,u_s\,\|S\|.
\end{gathered}
\label{eq:backward_stable}
\end{equation}

All our implementations of mixed precision GADI in section~\ref{sec:ne} use the Conjugate Gradient (CG) algorithm to solve the two 
GADI subsystems. However, note that CG is technically not
proven backward stable, at least in the usual sense, and therefore cannot be 
shown to always meet our assumptions~\eqref{eq:backward_stable}. Yet, under 
reasonable assumptions on the maximum number of CG iterations and the norms of
the computed CG iterates, CG will eventually compute solutions 
satisfying~\eqref{eq:backward_stable}; see~\cite[Theorem~2]{bcm25}. 

\subsection{Convergence of the forward error}
We first tackle the convergence of the forward error~\eqref{eq:429833} of mixed
precision GADI. Before presenting our main Theorem~\ref{thm2}, we recall some
results that we will use in our analysis. We present them in the following
lemmata.

\begin{lemma}[\cite{doi:10.1137/17M1140819} Section 3]
For \( A \in \mathbb{R}^{n \times n} \), \( b \in \mathbb{R}^n \), 
    \( \widehat{x}^{(k)}\in
    \mathbb{R}^n \), and $\widehat{y}^{(k)} \in \mathbb{R}^n$,
    the residual \( \widehat{r}^{(k)} \) and the next iterate
    $\widehat{x}^{(k+1)}$ computed, respectively, at line~\ref{line:alg1-res}
    and~\ref{line:alg1-update} of Algorithm~\ref{alg1} satisfy
\begin{align}
    \widehat{r}^{(k)} & = b - A\widehat{x}^{(k)} + \Delta \widehat{r}^{(k)}, \label{eq:r_k_def} \\
    \widehat{x}^{(k+1)} & = \widehat{x}^{(k)} + \widehat{y}^{(k)} + \Delta
    \widehat{x}^{(k)}, \label{eq:x_k_plus_1_def}
\end{align}
    with
\begin{align}
    \|\Delta \widehat{r}^{(k)}\| &\leq u_s \|b - A\widehat{x}^{(k)}\| + c(n)(1 +
    u_s)u_r (\|A\|\|\widehat{x}^{(k)}\| + \|b\|), \label{eq:r_k_bound} \\
    \|\Delta \widehat{x}^{(k)}\| &\leq  u  (\|\widehat{x}^{(k)}\| + \|\widehat{y}^{(k)}\|), \label{eq:x_k_plus_1_bound}
\end{align}
where $c(n)$ denotes a generic low-degree polynomial in the problem
dimension $n$. Note that $\Delta \widehat{r}^{(k)}$ also contains the error
    introduced while casting the residual from high precision $u_r$ to a lower
    precision $u_s$.
\label{lem:error_analysis}
\end{lemma}

\begin{lemma}[\cite{10.5555/500666} Corollary~4.19]\label{lem:384729_corrected}
Let \( B \in \mathbb{R}^{n \times n} \) be an invertible matrix, and let 
    \( F_k \in \mathbb{R}^{n \times n} \) be a perturbation matrix. Define the following
\begin{itemize}
  \item \( J_k \in \mathbb{R}^{n \times n} \), a matrix defined as \( J_k = (B
      + F_k)^{-1}B - I \).
  \item \( P_k \in \mathbb{R}^{n \times n} \), a matrix defined as \( P_k = B
      (B + F_k)^{-1} - I \).
\end{itemize}
Assume that the perturbation satisfies \( \|F_k\| \leq c(n) u_s  \|B\| \) and 
the condition $c(n) \kappa(B) u_s < 1/2$ holds, where $c(n)$ are low-degree
polynomials in the problem dimension $n$. Then
\begin{enumerate}
  \item The matrix \( B + F_k \) is nonsingular.
  \item The inverse of \( B + F_k \) can be expressed in two equivalent forms
  \begin{align*}
    (B + F_k)^{-1} = (I + J_k) B^{-1} = B^{-1} (I + P_k).
  \end{align*}
  \item The spectral norms of \(J_k\) and \(P_k\) satisfy
  \begin{align*}
    \|J_k\| \leq \frac{c(n) \, \kappa(B) \, u_s}{1 - c(n) \, \kappa(B) \, u_s} \leq 1, \quad
    \|P_k\| \leq \frac{c(n) \, \kappa(B) \, u_s}{1 - c(n) \, \kappa(B) \, u_s} \leq 1.
  \end{align*}
\end{enumerate}
\label{lem1_corrected}
\end{lemma}

We can now state the following theorem, which quantifies the convergence rate
(noted $\beta_{F}$)
and limiting accuracy (noted $\zeta_{F}$) of the forward errors of the mixed precision GADI
iterates.

\begin{theorem}
    Let Algorithm~\ref{alg1} be applied to the linear system \(Ax = b\), where
    \(A \in \mathbb{R}^{n \times n}\) is nonsingular. Assume that $H$ and
    $S$ defined in~\eqref{eq:split} satisfy assumption~\eqref{eq:293478}, and 
    that the solvers used at lines~\ref{line:alg1-sys1}
    and~\ref{line:alg1-sys2} of Algorithm~\ref{alg1} satisfy
    assumption~\eqref{eq:backward_stable}. 
    Then, for \(k \geq 0\), the computed 
    iterate \(\widehat{x}^{(k+1)}\) satisfies
    \begin{equation}\label{eq:342873}
        \|x - \widehat{x}^{(k+1)}\|  \lesssim \beta_F \|x -
        \widehat{x}^{(k)}\| + \zeta_F \|x\|,
    \end{equation}
    with
    \begin{align*}
        \beta_F & = \lambda_F^{(k)} + c(n)c_F \min\Big\{
        \kappa(HS)\kappa(H), \kappa(A)\big(\kappa(H) + \kappa(S)\big),
        \kappa(H)\kappa(S) \Big\}u_s \\
    & \quad + c(n)c_F \mu^{(k)}\kappa(A) u_s + c(n)c_F\kappa(A)u_r + c(n)c_F u, \\
        \zeta_F & = c(n) \Big(u + c_F \kappa(A) u_r \Big),
    \end{align*}
    where \(\lambda_F^{(k)} = \|T_F(\alpha, \omega) (x - \widehat{x}^{(k)}) \| /
    \norm{x - \widehat{x}^{(k)}}\), $T_F(\alpha, \omega)$ is the GADI iteration
    matrix defined in
    section~\ref{subsec:gadi_framework}, $\mu^{(k)}$ is defined in~\eqref{eq:947298}, and
    \(c_F = \|I - T_F(\alpha, \omega)\|\). 
    \label{thm2}
\end{theorem}
\begin{proof}~
To form~\eqref{eq:342873}, we wish to bound the error at the $(k+1)$th 
    iteration $\|x - \widehat{x}^{(k+1)}\|$ in
    terms of the error at the $k$th iteration $\|x - \widehat{x}^{(k)}\|$. 
    Using~\eqref{eq:x_k_plus_1_def}, we have
\begin{align*}
    x - \widehat{x}^{(k+1)} = (x - \widehat{x}^{(k)}) - \widehat{y}^{(k)} - \Delta \widehat{x}^{(k)}.
\end{align*}
Using assumption~\eqref{eq:backward_stable}, which characterizes the accuracies
of the solvers at lines~\ref{line:alg1-sys1} and~\ref{line:alg1-sys2} of 
Algorithm~\ref{alg1}, we can write \(\widehat{y}^{(k)}\) as
    \begin{equation}
        \begin{aligned}
    \widehat{y}^{(k)} = (2-\omega)\alpha (S+F_k^{\scriptscriptstyle S})^{-1}\widehat{z}^{(k)} =
            (2-\omega)\alpha (S+F_k^{\scriptscriptstyle S})^{-1}(H+F_k^{\scriptscriptstyle H})^{-1} \widehat{r}^{(k)},
    \label{eq:y_k_expression}             
\end{aligned}
    \end{equation}
where \(F_k^{\scriptscriptstyle S}\) and \(F_k^{\scriptscriptstyle H}\) are the perturbation matrices defined in~\eqref{eq:backward_stable}.
Note that $S$, \(F_k^{\scriptscriptstyle S}\), $H$, and \(F_k^{\scriptscriptstyle
    H}\) meet the conditions for the application of Lemma~\ref{lem1_corrected} by
    assumptions~\eqref{eq:293478} and~\eqref{eq:backward_stable}. Hence, by applying
    Lemma~\ref{lem1_corrected} with \(S\) and \(F_k^{\scriptscriptstyle S}\)
    yielding \(J_k^{\scriptscriptstyle S}\) and with \(H\) and
    \(F_k^{\scriptscriptstyle H}\) yielding \(P_k^{\scriptscriptstyle H}\), 
    and using~\eqref{eq:r_k_def}, we can further expand the expression of the 
    error
\begin{align*}
    x - \widehat{x}^{(k+1)} & = (x - \widehat{x}^{(k)}) -
    \alpha(2-\omega)(I+J_k^{\scriptscriptstyle S})(HS)^{-1}(I+P_k^{\scriptscriptstyle H})\widehat{r}^{(k)} - \Delta \widehat{x}^{(k)} \\
& = (x - \widehat{x}^{(k)}) - \alpha(2-\omega)(I+J_k^{\scriptscriptstyle
    S})(HS)^{-1}(I+P_k^{\scriptscriptstyle H})(b -
    A\widehat{x}^{(k)} + \Delta\widehat{r}^{(k)}) - \Delta \widehat{x}^{(k)}.
\end{align*}
Since \(b = Ax\), we can write \(b - A\widehat{x}^{(k)} = A(x - \widehat{x}^{(k)})\). Rearranging the terms yields
\begin{align*}
    x - \widehat{x}^{(k+1)} & = \Big(I -
    \alpha(2-\omega)(I+J_k^{\scriptscriptstyle
    S})(HS)^{-1}(I+P_k^{\scriptscriptstyle H})A\Big)(x - \widehat{x}^{(k)}) \\
    & \quad - \alpha(2-\omega)(I+J_k^{\scriptscriptstyle S})(HS)^{-1}(I+P_k^{\scriptscriptstyle H})\Delta\widehat{r}^{(k)} - \Delta \widehat{x}^{(k)} \\
    & = \Big(I - \alpha(2-\omega)(HS)^{-1}A\Big)(x - \widehat{x}^{(k)}) - \Delta \widehat{x}^{(k)}  
      - \alpha(2-\omega)(I+J_k^{\scriptscriptstyle S})(HS)^{-1}
     (I+P_k^{\scriptscriptstyle H})\Delta\widehat{r}^{(k)} \\
     & \quad - \alpha(2-\omega)\Big( (I+J_k^{\scriptscriptstyle S})(HS)^{-1}P_k^{\scriptscriptstyle H}A + J_k^{\scriptscriptstyle S}(HS)^{-1}A\Big)(x - \widehat{x}^{(k)}).
\end{align*}
By remarking that 
\begin{align}
    (HS)^{-1} = \frac{I - T_F(\alpha,\omega)}{\alpha(2-\omega)}A^{-1},
              \label{eq:HS_inverse}
\end{align}
where \(T_F(\alpha,\omega) = (HS)^{-1} (\alpha^2 I + MN - (1 - \omega) \alpha
A)\) (see section~\ref{subsec:gadi_framework}),
we can identify the iteration matrix $T_F(\alpha,\omega)$ in the expression of
the error $x - \widehat{x}^{(k+1)}$. We obtain
    \begin{equation}\label{eq:ineq1}
\begin{split}
    x - \widehat{x}^{(k+1)} & = T_F(\alpha,\omega)(x - \widehat{x}^{(k)}) - \Delta \widehat{x}^{(k)} - \alpha(2-\omega)(I+J_k^{\scriptscriptstyle
                       S})(HS)^{-1}(I+P_k^{\scriptscriptstyle H})\Delta\widehat{r}^{(k)}\\
                       & \quad - \alpha(2-\omega)\Big(
                       (I+J_k^{\scriptscriptstyle S})(HS)^{-1}P_k^{\scriptscriptstyle H}A +
                       J_k^{\scriptscriptstyle S}(HS)^{-1}A\Big)(x -
                       \widehat{x}^{(k)}).
\end{split}
    \end{equation}

We first consider the term $(I+J_k^{\scriptscriptstyle S})(HS)^{-1}P_k^{\scriptscriptstyle H}A + 
J_k^{\scriptscriptstyle S}(HS)^{-1}A$ in the error expression above~\eqref{eq:ineq1}. Using
the bounds on $\norm{J_k^{\scriptscriptstyle S}}$ and 
$\norm{P_k^{\scriptscriptstyle H}}$ in Lemma~\ref{lem1_corrected}, using the fact
that $\kappa(S) \leq \kappa(HS)\kappa(H)$, observing that $\norm{I +
J_k^{\scriptscriptstyle S}} \leq 2$, and denoting $c_F = 
\norm{I-T_F(\alpha, \omega)} = \alpha(2-\omega) \norm{(HS)^{-1}A}$,
we write
\begin{equation}\label{eq:395348}
\begin{split}
    \norm{(I+J_k^{\scriptscriptstyle
    S})(HS)^{-1}P_k^{\scriptscriptstyle H}A + J_k^{\scriptscriptstyle
    S} (HS)^{-1}A} 
    & = \norm{(I+J_k^{\scriptscriptstyle S})(HS)^{-1}P_k^{\scriptscriptstyle H} HS (HS)^{-1}A +
    J_k^{\scriptscriptstyle S}(HS)^{-1}A}\\
    & \lesssim
    c(n)\Big(\kappa(HS)\kappa(H) + \kappa(S)\Big) \norm{(HS)^{-1}A} u_s \\
    & \leq c(n)\frac{c_F}{\alpha(2-\omega)} \kappa(HS)\kappa(H) u_s.
\end{split}
\end{equation}
Or, alternatively, using $\norm{(HS)^{-1}}\norm{A}\leq 
c_F \kappa(A) / (\alpha(2-\omega))$, we can also bound this quantity as
\begin{equation}\label{eq:394278}
\begin{split}
    \norm{(I+J_k^{\scriptscriptstyle S})(HS)^{-1}P_k^{\scriptscriptstyle H}A + J_k^{\scriptscriptstyle
    S}(HS)^{-1}A} & \leq \norm{(HS)^{-1}}\norm{A} 
    \Big((\norm{I}+ \norm{J_k^{\scriptscriptstyle S}})\norm{P_k^{\scriptscriptstyle H}} + \norm{J_k^{\scriptscriptstyle
    S}}\Big) \\
    & \lesssim c(n)\frac{c_F}{\alpha(2-\omega)} \kappa(A)\Big(\kappa(H) +
    \kappa(S)\Big) u_s.
\end{split}
\end{equation}
Yet another alternative bound is possible by remarking that $H^{-1}
P_k^{\scriptscriptstyle H} H = J_k^{\scriptscriptstyle H}$, where
$J_k^{\scriptscriptstyle H} = (H+F_k^{\scriptscriptstyle H})^{-1} H - I$ 
from Lemma~\ref{lem1_corrected}. We then obtain
\begin{equation}\label{eq:233487}
\begin{split}
    \norm{(I+J_k^{\scriptscriptstyle S})(HS)^{-1}P_k^{\scriptscriptstyle H}A + J_k^{\scriptscriptstyle
    S}(HS)^{-1}A} & \leq \norm{(I+J_k^{\scriptscriptstyle S})S^{-1}(H^{-1}P_k^{\scriptscriptstyle H}H)S
    (HS)^{-1}A + J_k^{\scriptscriptstyle S}(HS)^{-1}A} \\
    & \leq \norm{(I+J_k^{\scriptscriptstyle S})S^{-1}J_k^{\scriptscriptstyle H}S + J_k^{\scriptscriptstyle
    S}} \norm{(HS)^{-1}A} \\
    & \lesssim c(n)\frac{c_F}{\alpha(2-\omega)} \kappa(H)\kappa(S) u_s.
\end{split}
\end{equation}

We now focus on the term $(I+J_k^{\scriptscriptstyle S})(HS)^{-1}(I+ 
P_k^{\scriptscriptstyle H})\Delta\widehat{r}^{(k)}$ in the error expression~\eqref{eq:ineq1}.
From Lemma~\ref{lem1_corrected}, we know that $\norm{I+J_k^{\scriptscriptstyle
S}}\leq 2$ and $\norm{I+P_k^{\scriptscriptstyle H}} \leq 2$. In addition, by
remarking that
\begin{align*}
    & \|A\|\|\widehat{x}^{(k)}\| + \|b\| =
\|A\|\|\widehat{x}^{(k)}\| + \|Ax\| \leq \|A\|\|\widehat{x}^{(k)} - x\| +
2\|A\|\|x\|,
\end{align*}
using~\eqref{eq:947298}, \eqref{eq:r_k_bound}, and $\norm{(HS)^{-1}}\norm{A}\leq 
c_F \kappa(A) / (\alpha(2-\omega))$, we obtain
\begin{equation}\label{eq:347938}
\begin{split}
    \norm{(I+J_k^{\scriptscriptstyle S})(HS)^{-1}(I+P_k^{\scriptscriptstyle H})\Delta\widehat{r}^{(k)}} 
    & \leq \norm{I+J_k^{\scriptscriptstyle S}} \norm{I+P_k^{\scriptscriptstyle
    H}} \norm{(HS)^{-1}} \norm{\Delta\widehat{r}^{(k)}} \\
    & \lesssim c(n) \norm{(HS)^{-1}} \Big(u_s \norm{b-A\widehat{x}^{(k)}} +
    u_r(\norm{A}\norm{\widehat{x}^{(k)}} + \norm{b}) \Big) \\
    & \leq c(n)\frac{c_F}{\alpha(2-\omega)} \kappa(A) \Big( (\mu^{(k)} u_s + u_r) \norm{x
    - \widehat{x}^{(k)}} + u_r \norm{x}\Big).
\end{split}
\end{equation}

Finally, we investigate the term $\Delta \widehat{x}^{(k)}$ in the error
expression~\eqref{eq:ineq1} and defined in Lemma~\ref{lem:error_analysis}.
To proceed, we need to bound \(\|\widehat{y}^{(k)}\|\).
Using~\eqref{eq:r_k_def}, \eqref{eq:r_k_bound}, \eqref{eq:y_k_expression},
Lemma~\ref{lem1_corrected}, and dropping second-order terms, we have
\begin{align*}
    \|\widehat{y}^{(k)}\| 
    \lesssim c(n) \alpha(2-\omega)\|(HS)^{-1}A\|\|x -
    \widehat{x}^{(k)}\| \leq c(n) c_F \|x - \widehat{x}^{(k)}\|.
\end{align*}
Using the above alongside~\eqref{eq:x_k_plus_1_bound}, we obtain
\begin{equation}\label{eq:432090}
    \norm{\Delta \widehat{x}^{(k)}} \leq u \Big(\norm{x - \widehat{x}^{(k)}} +
    \norm{x} +
    \norm{\widehat{y}^{(k)}}\Big) \lesssim u \norm{x} + c(n) c_F u
    \|x - \widehat{x}^{(k)}\|.
\end{equation}

We can now bound the norm of the error at the $(k+1)$th 
iteration~\eqref{eq:ineq1}. Combining~\eqref{eq:395348}, \eqref{eq:394278}, 
\eqref{eq:233487}, \eqref{eq:347938}, and~\eqref{eq:432090}, and denoting 
$\lambda_F^{(k)} = \|T_F(\alpha, \omega) (x - \widehat{x}^{(k)}) \| /
    \norm{x - \widehat{x}^{(k)}}$, we get
\begin{align*}
    \norm{x - \widehat{x}^{(k+1)}} & \lesssim \bigg( \lambda_F^{(k)} + c(n)c_F \min\Big\{
        \kappa(HS)\kappa(H), \kappa(A)\big(\kappa(H) + \kappa(S)\big),
        \kappa(H)\kappa(S) \Big\}u_s \\
    & \quad + c(n)c_F \mu^{(k)}\kappa(A)u_s + c(n)c_F\kappa(A)u_r + c(n)c_F u
    \bigg)
    \norm{x-\widehat{x}^{(k)}} + \bigg(u + c(n)c_F\kappa(A)u_r\bigg) \norm{x}.
\end{align*}
The previous bound simplifies to 
\(\|x - \widehat{x}^{(k+1)}\| \leq \beta_F\|x - \widehat{x}^{(k)}\| + \zeta_F\|x\|\), 
where \(\beta_F\) and \(\zeta_F\) are defined as in the theorem statement.
\end{proof}

%

Theorem~\ref{thm2} quantifies the contraction (or convergence rate) $\beta_F$ 
of the forward error at each GADI iteration. If $\beta_F< 1$, the forward error
is guaranteed to reduce until it reaches the limiting accuracy level 
$\zeta_F$. Within $\beta_F$, the term $\lambda_F^{(k)}$ is of particular
interest. We specifically require $\lambda_F^{(k)} < 1$ for all iterations
$k\geq i$ from a certain
iteration $i\geq 0$ for the GADI forward error to eventually converge. This necessary 
condition is straightforwardly met in the case where the 
highest singular value of $T_F(\alpha, \omega)$ is no more than $1$ since 
$\lambda_F^{(k)} \leq \norm{T_F(\alpha, \omega)}$. However, both in inexact and
exact arithmetic, GADI may converge even if $\norm{T_F(\alpha, \omega)} \geq
1$, so that the bound $\lambda_F^{(k)} \leq \norm{T_F(\alpha, \omega)}$ does
not fully describe the convergence. Instead, to better capture the convergence 
properties of GADI, we need to consider the asymptotic behavior of 
$\lambda_F^{(k)}$.

\begin{remark}
\label{rem:lambda_F_asymptotic}
In exact arithmetic, the error evolves as $x - x^{(k+1)} = 
T_F(\alpha,\omega)(x - x^{(k)})$, and its norm is reduced by a factor 
$\lambda_F^{(k)} = \|T_F(\alpha, \omega)(x - x^{(k)})\| / \|x - x^{(k)}\|$ at 
each iteration. As the iterations proceed, the error vectors will align with
the dominant invariant subspace of $T_F(\alpha,\omega)$ under mild conditions 
(e.g., the initial error $x-x^{(0)}$ should have a non-zero component in the dominant 
eigenvector of $T_F(\alpha, \omega)$). This means that 
$T_F(\alpha, \omega) (x - x^{(k)}) \approx \rho(T_F(\alpha, \omega)) (x -
x^{(k)})$ as $k$ increases; see~\cite[Section~4.2.1]{saad2003iterative}. 
Thus, $\lambda_F^{(k)}$ asymptotically converges to the spectral radius
$\rho(T_F(\alpha,\omega))$, which we know is lower than $1$ by 
Lemma~\ref{lem:gadi_convergence}, and which guarantees the convergence of GADI.

In the presence of rounding errors, additional difficulties arise for studying
the asymptotics of $\lambda_F^{(k)}$. We conjecture that, in its asymptotic regime,
$\lambda_F^{(k)}$ is bounded by
\begin{align}
        \lambda_F^{(k)} & \lesssim \rho\big(T_F(\alpha, \omega)\big) + c(n)c_F \min\Big\{
        \kappa(HS)\kappa(H), \kappa(A)\big(\kappa(H) + \kappa(S)\big),
        \kappa(H)\kappa(S) \Big\}u_s \nonumber\\
         & \quad + c(n)c_F \mu^{(k)}\kappa(A) u_s + c(n)c_F\kappa(A)u_r +
         c(n)c_F u \label{eq:943802}\\
        & \quad + c(n) \Big(u + c_F \kappa(A) u_r \Big) /\norm{x -
        \widehat{x}^{(k)}}, \nonumber
\end{align}
assuming that the number of GADI iterations $k$ is reasonable.  
Hence, if the perturbations induced by the rounding errors are negligible against
$\rho(T_F(\alpha, \omega))$, $\lambda_F^{(k)}$ converges
approximately toward $\rho(T_F(\alpha, \omega))$ as we iterate.
While we recognize that proving this statement is not trivial, we choose not to
attempt a proof in this article since this result is not critical for our 
study, and a rigorous development cannot be made concisely. We think, however, 
that a proof is possible by adapting derivations 
in~\cite[sect.~17.2]{10.5555/579525} to express $x - \widehat{x}^{(k)}$ in 
terms of
$T_F^k(\alpha, \omega)(x - x^{(0)})$, which should then be used to bound
$\lambda_F^{(k)} = \|T_F(\alpha, \omega)(x - \widehat{x}^{(k)})\| / \|x -
\widehat{x}^{(k)}\|$ as in~\eqref{eq:943802}.
\end{remark}

Thus, under Remark~\ref{rem:lambda_F_asymptotic}, the convergence rate of the
forward error is likely to behave as
\begin{equation}\label{eq:234790}
    \begin{split}
        \beta_F & = \rho\big(T_F(\alpha, \omega)\big) + c(n)c_F \min\Big\{
        \kappa(HS)\kappa(H), \kappa(A)\big(\kappa(H) + \kappa(S)\big),
        \kappa(H)\kappa(S) \Big\}u_s \\
    & \quad + c(n)c_F \mu^{(k)}\kappa(A) u_s + c(n)c_F\kappa(A)u_r + c(n)c_F u,
    \end{split}
    \end{equation}
after a certain number of iterations. We observed throughout our experiments
that it often just takes a few first iterations for the convergence rate 
$\beta_F$ to reach its asymptotic regime~\eqref{eq:234790}.

\subsection{Convergence of the backward error}

We now study the convergence of the backward error~\eqref{eq:231477} of mixed
precision GADI. Similarly to Theorem~\ref{thm2} on the forward error, we 
quantify the convergence rate (noted $\beta_B$) and limiting accuracy (noted 
$\zeta_B$) of the backward error in the following theorem. 
\begin{theorem}
\label{thm3}
Under the hypotheses of Theorem~\ref{thm2}, the computed iterate
$\widehat{x}^{(k+1)}$ satisfies for all $k\geq 0$
\begin{equation}
    \|b-A\widehat{x}^{(k+1)}\| 
    \lesssim \beta_{B} \|b-A\widehat{x}^{(k)}\|   
    + \zeta_B\Big( \|A\| \|\widehat{x}^{(k)}\|+\|b\| \Big),
    \label{eq:be_recur_local}
\end{equation}
with
\begin{align*}
    \beta_{B} & =  \lambda^{(k)}_{B} + c(n) c_B 
    \min\Big\{\kappa(HS)\,\kappa(S), \kappa(A)\big(\kappa(H)+\kappa(S)\big),
    \kappa(H)\kappa(S)\Big\} u_s \\
    & \quad + c(n) c_B\kappa(A) u ,\\
    \zeta_B & = c(n) \Big( u + c_B u_r \Big),
\end{align*}
where $\lambda^{(k)}_{B} = \|T_B (b-A\widehat x^{(k)})\| / \|b-A\widehat
    x^{(k)}\|$, $T_{B}(\alpha, \omega)$ is the similarity transform of the
    iteration matrix defined in section~\ref{subsec:gadi_framework}, 
    and $c_B = \|I-T_B(\alpha,\omega)\|$.
\end{theorem}
\begin{proof}~
Identically to the proof of Theorem~\ref{thm2}, we wish to bound the residual
at the $(k+1)$th iteration $\norm{b - A\widehat{x}^{(k+1)}}$ in terms of
the residual at the $k$th iteration $\norm{b - A\widehat{x}^{(k)}}$. 
Applying $A$ to the left in~\eqref{eq:ineq1}, we obtain
    \begin{align}\label{eq:234782}
    b - A\widehat{x}^{(k+1)}
    & = A T_F(\alpha,\omega)\,(x-\widehat{x}^{(k)})
         - A \Delta \widehat{x}^{(k)} - \alpha(2-\omega)
         A\,(I+J_k^{\scriptscriptstyle S})(HS)^{-1}(I+P_k^{\scriptscriptstyle
         H}) \Delta\widehat{r}^{(k)} \nonumber \\
    & \quad - \alpha(2-\omega) A\Big((I+J_k^{\scriptscriptstyle S})(HS)^{-1}P_k^{\scriptscriptstyle H}A + J_k^{\scriptscriptstyle S}(HS)^{-1}A\Big)(x-\widehat{x}^{(k)}).
\end{align}

We can rework the above expression using the similarity transform
$T_B(\alpha, \omega) = A T_F(\alpha, \omega)A^{-1}$. In particular, 
using~\eqref{eq:HS_inverse} and $b-A\widehat x^{(k)} = A(x-\widehat
x^{(k)})$, we have
    \begin{equation}\label{eq:394783}
    \begin{gathered}
    A(HS)^{-1} = \frac{I - T_B(\alpha, \omega)}{\alpha(2-\omega)}, \qquad
    \|A(HS)^{-1}\| = \frac{c_B}{\alpha(2-\omega)}, \\
    \mbox{and} \quad AT_F(\alpha,\omega)(x-\widehat x^{(k)}) = T_B(b-A\widehat x^{(k)}).
\end{gathered}
    \end{equation}

    Hence, with~\eqref{eq:394783}, noting that $\kappa(H) \leq
    \kappa(HS)\kappa(S)$, and using a very similar derivation 
to~\eqref{eq:395348},~\eqref{eq:394278}, and~\eqref{eq:233487} in the proof of
    Theorem~\ref{thm2}, we can bound
    the term $\alpha(2-\omega) A((I+J_k^{\scriptscriptstyle
    S})(HS)^{-1}P_k^{\scriptscriptstyle H}A + J_k^{\scriptscriptstyle
    S}(HS)^{-1}A)(x-\widehat{x}^{(k)})$ in~\eqref{eq:234782} with
\begin{align}
    & \norm{A\Big((I+J_k^{\scriptscriptstyle S})(HS)^{-1}P_k^{\scriptscriptstyle H}A + J_k^{\scriptscriptstyle S}(HS)^{-1}A\Big)(x-\widehat{x}^{(k)})}
     = \norm{\Big(A(I+J_k^{\scriptscriptstyle
     S})(HS)^{-1}P_k^{\scriptscriptstyle H} + A J_k^{\scriptscriptstyle
     S}(HS)^{-1}\Big)(b-A\widehat{x}^{(k)})} \nonumber \\
     & \qquad \qquad \lesssim c(n) c_B \min\Big\{\kappa(HS)\,\kappa(S),
     \kappa(A)\big(\kappa(H)+\kappa(S)\big),
     \kappa(H)\,\kappa(S)\Big\}\,u_s\,\|b-A\widehat x^{(k)}\|.
     \label{eq:230429}
\end{align}
Note a minor difference with the result of Theorem~\ref{thm2}, where we now obtain
a term $\kappa(HS) \kappa(S)$ in the ``min'' instead of $\kappa(HS) \kappa(H)$.

Using~\eqref{eq:HS_inverse}, \eqref{eq:394783}, Lemmata~\ref{lem:error_analysis} and~\ref{lem1_corrected}, the 
expression~\eqref{eq:y_k_expression} of $\widehat{y}^{(k)}$, and dropping 
second-order terms, we can bound the norms of the terms involving 
$\Delta\widehat r^{(k)}$ and $\Delta\widehat x^{(k)}$ in~\eqref{eq:234782}.
We have
    \begin{equation}\label{eq:314097}
\begin{split}
    \norm{A(I+J_k^{\scriptscriptstyle S})(HS)^{-1}(I+P_k^{\scriptscriptstyle H})\Delta\widehat{r}^{(k)}} 
    & \lesssim c(n)\frac{c_B}{\alpha(2-\omega)} \Big( u_s \norm{b-A\widehat{x}^{(k)}} +
    u_r(\norm{A}\norm{\widehat{x}^{(k)}} + \norm{b})\Big),
\end{split}
\end{equation}
and
\begin{equation}\label{eq:384790}
    \norm{A\Delta \widehat{x}^{(k)}} \lesssim u \Big(\norm{A}\norm{\widehat{x}^{(k)}} +
    c(n) c_B \kappa(A) \|b - A\widehat{x}^{(k)}\|\Big).
\end{equation}

Finally, taking the norm of~\eqref{eq:234782}, collecting the 
bounds~\eqref{eq:230429}, \eqref{eq:314097}, and~\eqref{eq:384790}, and 
remarking that $\norm{AT_F(\alpha,\omega)(x-\widehat x^{(k)})} =
\lambda^{(k)}_B\norm{b-A\widehat x^{(k)}}$ with $\lambda^{(k)}_B$ defined in
the theorem statement, we
recover~\eqref{eq:be_recur_local}.
\end{proof}

Under analogous arguments to Remark~\ref{rem:lambda_F_asymptotic}, which apply
similarly to $\lambda^{(k)}_{B} = \|T_B (b-A\widehat x^{(k)})\| / \|b-A\widehat
x^{(k)}\|$, we conclude that the asymptotic regime of the backward error
convergence rate satisfies
\begin{equation}
    \begin{split}
        \beta_{B} & \approx \rho\big(T_B(\alpha, \omega)\big) + c(n) c_B 
    \min\Big\{\kappa(HS)\,\kappa(S), \kappa(A)\big(\kappa(H)+\kappa(S)\big),
    \kappa(H)\kappa(S)\Big\} u_s \\
    & \quad + c(n) c_B\kappa(A) u.
    \end{split}
\end{equation}

\subsection{Comparison to previous rounding error analyses}

\paragraph{Comparison to the mixed precision IR error analysis.}
The forms of Theorems~\ref{thm2} and~\ref{thm3} are in some ways similar to the
classic IR results~\cite[Theorems~3.2 and~4.1]{doi:10.1137/17M1140819}, 
respectively. Due to the proximity of the methods, it is natural to compare
our mixed precision GADI convergence results against those of IR. To ease our
commentary, we simplify the expressions of the convergence rates and limiting 
accuracies with the following assumptions:
\begin{itemize}
    \item The splitting matrices $H$ and $S$ are at most as ill-conditioned as 
        $A$ or $HS$ (i.e., $\max\{\kappa(S), \kappa(H)\} \leq \min\{\kappa(A),
        \kappa(HS)\}$), which we cannot guarantee under our general 
        assumptions, but which tend to be true for many practical splittings.
    \item The precision $u$ and $u_r$ are chosen high enough so that we satisfy
        $\max\{\kappa(A)u, \kappa(A)u_r\} \leq \kappa(H)\kappa(S)u_s$.
    \item The term $\mu^{(k)}$ is of order $\kappa(A)^{-1}$ so that 
        $\mu^{(k)}\kappa(A) \leq \kappa(H)\kappa(S)$. We explained in
        section~\ref{sec:notations} that we can expect $\mu^{(k)}\approx
        \kappa(A)^{-1}$ on the first iterations of GADI. It generally remains
        true up until the backward error reaches its limiting accuracy.
    \item The norms $\norm{T_F(\alpha,\omega)}$ and $\norm{T_B(\alpha,\omega)}$
        are of modest size so that $c_F$ and $c_B$ are small constants that can
        be accumulated in $c(n)$.
\end{itemize}
Under these assumptions and using $u_s \leq u \leq u_r$, the convergence rates and limiting accuracies
reduce to
\begin{equation}\label{eq:347289}
\begin{gathered}
    \beta_F = \lambda_F^{(k)} + c(n) \kappa(H)\kappa(S) u_s \quad \mbox{and} \quad \zeta_F
    = c(n) \Big(u+\kappa(A)u_r\Big), \\
    \beta_B = \lambda_{B}^{(k)} + c(n) \kappa(H) \kappa(S) u_s \quad \mbox{and}\quad \zeta_B = c(n) u.
\end{gathered}
\end{equation}

Hence, equivalently to IR, the convergence rates of mixed precision GADI are
governed by the precision $u_s$, and the limiting accuracies by the precisions
$u$ and $u_r$. Importantly, this offers the same opportunities as IR: 
\begin{itemize}
    \item We can solve the GADI subsystems inaccurately and cheaply while still
        computing a highly accurate solution. 
    \item The backward error limiting accuracy $\zeta_B$ is independent of
        $\kappa(A)$, where the forward error counterpart $\zeta_F$ is. In other
        words, if $u = u_r$ and the system is ill-conditioned, $\zeta_F$ can be
        significantly higher than $\zeta_B$.
    \item The effect of $\kappa(A)$ on the limiting accuracy of the forward error
        $\zeta_F$ can be removed by setting $u_r$ in extra precision 
        (e.g., $u_r = u^2$). 
\end{itemize}

While similar to IR to an extent, our GADI convergence results also present 
important differences worth noticing. 

First, the convergence rates $\beta_F$ 
and $\beta_B$ feature the terms $\lambda_F^{(k)}$ and $\lambda_B^{(k)}$, 
respectively. These terms, as we explained, will eventually converge to the 
spectral radii $\rho(T_{F}(\alpha, \omega)) = \rho(T_{B}(\alpha, \omega))< 1$
under mild conditions; see Remark~\ref{rem:lambda_F_asymptotic}. Importantly, the spectral radii
are (almost) 
independent from the choice of precisions $u$, $u_r$, and $u_s$, and often
range from $0.1$ up to $0.95$ in practice. Consequently, GADI offers a strong incentive to use very 
aggressive low precision $u_s$, so that $c(n) \kappa(H)\kappa(S) u_s \approx
\rho(T_{F}(\alpha, \omega)) = \rho(T_{B}(\alpha, \omega))$.
Indeed, having $c(n) \kappa(H)\kappa(S) u_s$ 
significantly lower than
$\rho(T_{F}(\alpha, \omega)) = \rho(T_{B}(\alpha, \omega))$ will not improve 
the convergence rates $\beta_F$ and $\beta_B$, which means that there are few
benefits to keeping $u_s$ in high precision.
This last assertion should, however, be slightly mitigated. As we will explain
in section~\ref{sec:pr}, the choice of parameter $\alpha$ is not fully 
independent from the precision $u_s$, and can significantly
influence the convergence rate of GADI. Additionally, while lowering the
precision $u_s$ may not affect the GADI convergence rate, if
lines~\ref{line:alg1-sys1} and~\ref{line:alg1-sys2} of Algorithm~\ref{alg1} are
solved with some iterative solvers (e.g., CG), setting $u_s$ in low precision may increase
the number of inner iterations required by these solvers.

Second, the ``rounding errors induced parts'' of the convergence rates 
$\beta_F$ and $\beta_B$ are controlled by $c(n) \kappa(H)\kappa(S) u_s$, where a
classic IR using a backward stable direct solver has a convergence rate driven 
by $c(n)\kappa(A) u_s$; see~\cite[sect.~7]{doi:10.1137/17M1140819}. Therefore,
compared with IR, the convergence rate of GADI is not directly determined by
the condition number of $A$, but rather by the condition numbers of the
splitting matrices $H$ and $S$. We cannot definitively state which is better between
$c(n) \kappa(H)\kappa(S) u_s$ and $c(n)\kappa(A) u_s$ in the
general case without more assumptions on the splitting used. In theory, both
scenarios are possible: $\kappa(A) \leq \kappa(H)\kappa(S)$
and $\kappa(H)\kappa(S) \leq \kappa(A)$. However, this result showcases the
importance of a good choice of splitting when leveraging low precision: one
should try to construct a splitting that reduces or limits $\kappa(H)$ and
$\kappa(S)$.

\paragraph{Previous rounding error analyses of ADI methods.}

Previous rounding error analyses of ADI methods have
been proposed by Rachford~\cite{doi:10.1137/0705032} and 
Zawilski~\cite{Zawilski1992Optimal,Zawilski01011990}. 
Considering the fixed precision setup $u_s = u = u_r$, our analysis differs from
these previous results in three main ways.
First, unlike these previous studies that assume commuting and symmetric 
positive definite splittings~\cite{Zawilski1992Optimal,Zawilski01011990} or 
focus on specific PDE time-stepping schemes with particular matrix 
structures~\cite{doi:10.1137/0705032}, our results apply to 
general non-commuting and non-symmetric splittings.
Second, Rachford~\cite{doi:10.1137/0705032} analyses the accumulation 
of rounding errors in time-stepping procedures, and 
Zawilski~\cite{Zawilski1992Optimal,Zawilski01011990} studies the 
limiting accuracy of cyclic ADI iterations; that is, where parameters vary cyclically.
Neither addresses how rounding errors affect the convergence rate 
of the iterations, which we do in this work for stationary ADI iteration where 
parameters are fixed.
Third, regarding stability, prior analyses bound the forward error limiting
accuracy
by $\mathcal{O}(\kappa(A) u)$~\cite{Zawilski1992Optimal,Zawilski01011990}. While
they also provide bounds for the residual, these depend on $\kappa(A)$. In contrast, our
analysis proves that the backward error limiting accuracy $\zeta_B$ is of order $\mathcal{O}(u)$,
independent of the condition number $\kappa(A)$.

Perhaps the work by Bai and Rozlo{\v{z}}n{\'\i}k~\cite{bai2015numerical}
achieves results that are the closest to ours. Their analysis does not cover
directly ADI methods, but addresses one-step and two-step
stationary iterative methods using generic matrix splittings. In particular,
they show that stationary iteration can solve the subsystems inaccurately
while preserving a limiting backward error of order 
$\mathcal{O}(u)$ under mild conditions. GADI can be rewritten as a one-step
stationary iteration, so that the model in~\cite{bai2015numerical} could 
be used to determine limiting accuracies for the backward and forward errors of GADI. However, 
our analysis features improvements and additions that are critical to
understanding
the numerical behavior of mixed precision GADI, and which we think cannot be 
recovered from~\cite{bai2015numerical}. For instance, \cite{bai2015numerical}
does not account for the use of extra precision $u_r$ for computing the
residual at line~\ref{line:alg1-res} of Algorithm~\ref{alg1}, which enables 
GADI to achieve a better forward error limiting accuracy. Most importantly,
\cite{bai2015numerical} does not provide clear expressions for the convergence
rates $\beta_F$ and $\beta_B$. We also suspect that if we were to extrapolate
those expressions from the derivations in~\cite{bai2015numerical}, we would obtain significantly more
pessimistic convergence rates for mixed precision GADI.



\section{Parameter selection strategy}
\label{sec:pr}

The number of GADI iterations is very sensitive to the choice of the regularization
parameter $\alpha$. 
To make the mixed precision GADI framework practical and robust, we present in
this section the systematic method we employ for selecting this parameter.

\subsection{On the impact of \texorpdfstring{\(\alpha\)}{alpha} on the convergence}
\label{subsec:theory_alpha}

A good choice of the regularization parameter $\alpha$ is not straightforward.
First, it depends on the input problem. 
Second, $\alpha$ adjusts two antagonistic effects. Namely, on the one
hand, increasing $\alpha$ can improve the conditioning of \(H= \alpha I +M\)
and \(S= \alpha I + N\) since
it controls the magnitude of the diagonal shift, thereby lowering the term 
$\kappa(H)\kappa(S) u_s$ in the convergence rates~\eqref{eq:347289}.
On the other hand, $HS = \alpha^2 I + \alpha A + MN$ will eventually be 
dominated by $\alpha^2 I$ as $\alpha$ grows, leading to 
$\lim_{\alpha \rightarrow \infty} \norm{I - T_F(\alpha,\omega)} =
\lim_{\alpha \rightarrow \infty} \alpha(2-\omega)\norm{(HS)^{-1}A} = 0$
using~\eqref{eq:HS_inverse}. This means that \(T_F(\alpha,\omega)\)
approaches the identity as $\alpha$ grows and, thus, $\lim_{\alpha \rightarrow
\infty} \rho(T_F(\alpha,\omega)) = 1$, thereby increasing the term
$\lambda_F^{(k)}$ in the convergence rate~\eqref{eq:347289}. Naturally, the
same holds for $\lambda_B^{(k)}$.

\subsection{A Data-Driven Parameter Prediction}
\label{subsec:parameter_prediction}
We use the same Gaussian Process Regression (GPR) as 
in~\cite[Section~3]{doi:10.1137/21M1450197}, which we train
cheaply on problems of small sizes to predict a near-optimal
\(\alpha\) for larger instances.
This data-driven approach is particularly well-suited because the selection of an
optimal \(\alpha\) for mixed precision GADI depends nonlinearly on the problem 
size and the set of precisions $u_s$, $u$, and $u_r$.
Additionally, as we explained in the previous section~\ref{subsec:theory_alpha}, $\alpha$
adjusts different antagonistic effects for which it is difficult to find a 
trade-off empirically. We found GPR to be able to capture these complex
trade-offs and nonlinear trends efficiently, providing a
robust estimate (with uncertainty).

\subsection{A systematic selection strategy for \texorpdfstring{\(\alpha\)}{alpha}}
\label{subsec:alpha_selection}

It is important to automate the selection of a good parameter $\alpha$ for a
practical and black-box implementation of mixed precision GADI. We summarize
our procedure below:
\begin{enumerate}
    \item \textbf{Initialize \(\alpha\)} with the GPR predictor trained on a
        dataset of problems of smaller sizes; see section~\ref{subsec:parameter_prediction}.
        Optionally, $\alpha$ can be initialized from human input if there is
        prior knowledge on similar matrices and problems. 
    \item \textbf{(Optional) Validate that the theoretical convergence 
        condition is met} for this $\alpha$ by ensuring that
        \begin{equation}\label{eq:239479}
        \kappa(H)\kappa(S) u_s < \tau,
        \end{equation}
        for a chosen safety threshold $0<\tau<1$ (e.g., $\tau = 0.01$). 
        If ``cheap'' condition number estimators for $H$ and $S$
        are not available, this step can be skipped.
    \item If the condition~\eqref{eq:239479} holds, we start GADI with the current 
    $\alpha$. Otherwise, we 
    \textbf{increase \(\alpha\) to reduce the condition number of $H$ and
    $S$} until~\eqref{eq:239479} is met. If the condition~\eqref{eq:239479} 
    cannot be computed cheaply, we
    start GADI with the current $\alpha$. If we observe a stagnation or divergence of the
    backward error~\eqref{eq:231477}, we stop GADI, increase $\alpha$, and
    repeat the process.
\end{enumerate}

\section{Performance analysis on GPU}
\label{sec:ne}

In this section, we evaluate the performance of mixed precision GADI 
across various configurations of precisions and implemented on an
NVIDIA A100 GPU. We apply the algorithm to three distinct
problems: Two-Dimensional Convection-Diffusion-Reaction Equation, 
Three-Dimensional Convection-Diffusion Equation, and Complex 
Reaction-Diffusion Equation. For these three problems, we compare:
\begin{itemize}
    \item \textbf{GADI-FP64} - GADI in full double precision which is 
        Algorithm~\ref{alg1} using $u=u_r=u_s=$ FP64. The two shifted
        subsystems are solved with CG.
    \item \textbf{GADI-FP32} - Mixed precision GADI using $u=u_r=$ FP64 and CG
        in precision
        $u_s=$ FP32.
    \item \textbf{GADI-BF16} - Mixed precision GADI using $u=u_r=$ FP64 and CG
        in precision $u_s=$ BF16.
    \item \textbf{cuDSS} - The CUDA direct sparse solver in double or single 
        precision, which serves as a first baseline for our benchmark.
    \item \textbf{GMRES-FP32} - Mixed precision IR using the Generalized
        Minimal Residual method (GMRES) in single precision as in~\cite{lld22}
        (i.e., $u=u_r=$ FP64 and $u_s=$ FP32 in section~\ref{subsec:mp_ir}),
        and which serves as a second baseline.
\end{itemize}

\subsection{Implementation details and experimental setting}
\label{subsec:implement}

\paragraph{Implementations of GADI-FP64 and GADI-FP32.} For all the problems we
consider, each GADI iteration yields a shifted symmetric positive definite and
a shifted skew-symmetric subsystem to solve. Both subsystems are solved via an
in-house implementation of CG. For the skew-symmetric system we use CG on the 
normal equation. Each CG iteration uses one SpMV (two for the skew-symmetric 
system), two DOTs, and several AXPY/SCAL kernels, implemented with the cuBLAS 
and cuSPARSE libraries. The CG stopping criterion is a tolerance on the relative
residual $\|b-A\widehat{x}^{(k)}\|/\|r_0\|$. For each problem, we select the 
tolerance in the set $\{10^{-2}, 10^{-3}, 10^{-4}, 10^{-6}\}$ that leads to the
best performance. GADI-FP64 and GADI-FP32 both compute
lines~\ref{line:alg1-res} and~\ref{line:alg1-update} of Algorithm~\ref{alg1} in
FP64. GADI-FP64 uses CG in FP64, whereas GADI-FP32 leverages CG in FP32. We
use the zero vector $x_0=0$ as the initial guess.

\paragraph{Distinctive features of the GADI-BF16 implementation.} GADI-BF16 
uses the \texttt{cublas*Ex} and cuSPARSE APIs that can leverage BF16 
performance by decoupling storage type and compute type. Hence, with GADI-BF16,
the matrices and vectors are stored in BF16 but the computation is performed in
FP32. As our CG implementation is memory-bound, GADI-BF16 benefits from memory 
reduction and increased speedup due to the better BF16 bandwidth, as well as
increased stability due to the FP32 accumulation within the computing unit. We 
employ an early-stop strategy to detect and prevent stagnation.

\paragraph{GADI memory consumption.} Most of our GADI memory usage concerns 
storing the different matrix operators. Besides storing $A$, we precompute the
symmetric and skew-symmetric matrices $H$ and $S$, as well as the matrix $\alpha I - N$, and store them in memory. 
Note that this yields a significant memory increase that could be avoided by 
forming the applications of $H$, $S$ and $\alpha I - N$ to a vector by means of SpMVs with
$A$. On the other hand, by not storing these matrices explicitly, we would require
more SpMV calls and a significant runtime increase. For this reason, 
we choose to prioritize runtime at the cost of having to store $H$, $S$ and $\alpha I - N$.
With this implementation, by storing $H$, $S$ and $\alpha I - N$ in low precision $u_s$ (FP32 
or BF16), mixed precision GADI can achieve noticeable memory reduction over 
GADI-FP64.

\paragraph{Implementation of GMRES-FP32.} We implement the mixed-precision
GMRES-based IR from~\cite{lld22}, which is a form of restarted GMRES, and which
we use as a state-of-the-art sparse iterative solver baseline. The 
update of the iterate \(x^{(k+1)} = x^{(k)} + y^{(k)}\) and the residual
computation \(r^{(k)}=b-Ax^{(k)}\) are performed in FP64, while the inner GMRES iterations
compute and store in FP32. We employ Modified Gram-Schmidt for the orthogonalization and restart GMRES after 50 iterations. We
cross-checked configuration and timings on a subset of problems 
from~\cite{lld22} and tuned accordingly to obtain comparable (or better)
performance.

\paragraph{cuDSS baseline.} We use NVIDIA's cuDSS as a state-of-the-art sparse 
direct solver baseline. In our experiments, we use its FP64 or FP32 version
depending on the solution target accuracy. We report
the runtime and (estimated) memory consumption as a point of reference for 
performance. Note that, to our knowledge, cuDSS does not provide 
access to the factorization memory peak. Therefore, the memory we report can be
lower than the actual memory peak usage of the solver.

\paragraph{Gaussian Process Regression (GPR) for parameter selection.} We use
the systematic $\alpha$ selection strategy of section~\ref{sec:pr}. A training 
dataset is built once by extracting features and computing the optimal $\alpha$
on representative, smaller, cheaper instances. A lightweight GPR model is then
trained to predict $\alpha$ for larger problems. This model and its feature set 
are reused across all reported runs  of a given equation: dataset generation and training are a 
one-time cost. For this reason, we do not account for it in the
performance results presented in this section. We observed the execution time 
of the full GPR pipeline (i.e., dataset generation and training) to be within 
an acceptable range.

\paragraph{Measurement methodology.} Execution time is measured with CUDA 
events. Results are achieved after a warm-up run to avoid cold-start effects. 
Device memory is obtained from before/after snapshots; due to extra temporary 
memory usage like SpMV workspace, these numbers are lower bounds of the exact
memory peak consumptions.

\paragraph{Environment.} All experiments are run on a single NVIDIA A100 80GB 
SXM GPU with the following host setup: Intel(R) Xeon(R) Platinum 8358 @ 2.60GHz
CPU, 128GB DDR4 RAM, Ubuntu 22.04.6 LTS, CUDA 12.8. We use cuBLAS and cuSPARSE 
for dense and sparse kernels and cuDSS for the FP64 direct baseline. The 
compiler is nvcc.

\paragraph{Code availability.} The source code implementing the mixed precision 
GADI variants, the GMRES-FP32 IR, and the GPR-based parameter selection is available at 
\url{https://github.com/gejifeng/code-for-gadi-mix} and was used to generate the
reported results.

\subsection{Two-Dimensional Convection-Diffusion-Reaction Equation} \label{subsec:2de}

\paragraph{Problem description.}
We consider the two-dimensional steady-state convection-diffusion-reaction equation
\begin{equation*}
    -\left(u_{x_1x_1} + u_{x_2x_2}\right)
    + 2r\,(u_{x_1} + u_{x_2})
    + 100\,u = f(x_1, x_2),
\end{equation*}
defined on the unit square $\Omega = [0,1]^2$ with homogeneous Dirichlet boundary conditions.
This equation models the balance between diffusion, convection, and reaction processes, 
and is widely encountered in transport phenomena, fluid mechanics, and heat transfer applications \cite{morton1996numerical}.

Discretizing with centered finite differences on a uniform grid with spacing $h = 1/(n_g+1)$
yields a sparse linear system $A x = b$, where $A \in \mathbb{R}^{n_g^2 \times n_g^2}$ has the Kronecker product structure
\begin{align*}
    A = I \otimes T_x + T_y \otimes I.
\end{align*}
Here $T_x, T_y \in \mathbb{R}^{n_g \times n_g}$ are one-dimensional discrete operators representing
diffusion and convection along each spatial direction, defined as
\begin{align*}
    T_x = T_y = M + 2r\,N + \frac{100}{(n_g+1)^2}I,
\end{align*}
where $M = \mathrm{Tridiag}(-1, 2, -1), N = \mathrm{Tridiag}(0.5, 0, -0.5)$.
The tridiagonal matrix $M$ corresponds to the centered discretization of the second-order derivative,
while $N$ represents the first-order convective derivative operator.
The reaction term $100u$ contributes to the diagonal shift $\frac{100}{(n_g+1)^2}I$ after scaling by $h^2$.

The parameter $r$ controls the relative convection strength and we fix it to
$r=1.0$. The parameter $n_g$ denotes the number of grid points per dimension,
leading to a problem size of $n = n_g^2$. The resulting matrix $A$ is non-symmetric. 

Mixed precision GADI is applied with the symmetric/skew-symmetric splitting
\begin{align*}
H = \alpha I + \frac{A + A^T}{2}, \quad S = \alpha I + \frac{A - A^T}{2},
\end{align*}
This splitting is a variant of the Hermitian and skew-Hermitian splitting
(HSS) method, which was originally proposed for solving non-Hermitian positive 
definite linear systems \cite{bai2003hermitian}. The HSS iteration and its 
variants are prevalent solvers and preconditioners for solving linear 
systems derived from the discretization of convection-diffusion 
equations~\cite{Bertaccini2005,doi:10.1137/080723181}.

\paragraph{Performance comparison.}
We summarize the runtimes and estimations of the memory consumption of
the different solvers for increasing problem sizes in
Table~\ref{tab:2d_convection_diffusion_performance}. The largest linear system
has a dimension $n=10^8$.
All reported runs achieve a final relative residual accuracy of at least
$\norm{b-A\widehat{x}_{k}}/\norm{r_0} \leq 10^{-10}$.

\begin{table}[t!]
    \centering
    \caption{
    \label{tab:2d_convection_diffusion_performance}
    Runtime (seconds) and (estimated) memory usage (GiB) for the two-dimensional convection--diffusion--reaction
    equation using GADI-FP64/FP32/BF16, GMRES-FP32, and cuDSS-FP64. We report problems of increasing
    grid sizes $n_g \in \{960,\,2560,\,4096,\,8192,\,10000\}$, where $n=n_g^2$ is the system
    dimension. All methods achieve a relative residual of at least $\norm{b-A\widehat{x}_{k}} / \norm{r_0}
    \leq 10^{-10}$. A missing entry denoted by ``\textemdash'' indicates either
non-convergence within the prescribed iteration budget or device memory
    exhaustion.}
    \setlength{\tabcolsep}{4.4pt}
\begin{tabular*}{0.99\linewidth}{c | rrr rr | rrr rr}
        \toprule
    \multirow{3}{*}{\shortstack{Grid\\Size\\$n_g$}} &
        \multicolumn{5}{c |}{\textbf{Runtime Performance (seconds)}} &
        \multicolumn{5}{c}{\textbf{Memory (GiB)}} \\
        \cmidrule(lr){2-6} \cmidrule(lr){7-11}
        & \multicolumn{3}{c}{\textbf{GADI}} &
        \multicolumn{1}{c}{\textbf{GMRES}} & \multicolumn{1}{c |}{\textbf{CUDSS}}
        & \multicolumn{3}{c}{\textbf{GADI}} &
        \multicolumn{1}{c}{\textbf{GMRES}} &
        \multicolumn{1}{c}{\textbf{CUDSS}} \\
        \cmidrule(lr){2-4} \cmidrule(lr){5-5} \cmidrule(lr){6-6}
        \cmidrule(lr){7-9} \cmidrule(lr){10-10} \cmidrule(lr){11-11}
        & \multicolumn{1}{c}{FP64} & \multicolumn{1}{c}{FP32} &
        \multicolumn{1}{c}{BF16} & \multicolumn{1}{c}{FP32} &
        \multicolumn{1}{c |}{FP64} &
        \multicolumn{1}{c}{FP64} & \multicolumn{1}{c}{FP32} &
        \multicolumn{1}{c}{BF16} & \multicolumn{1}{c}{FP32} &
        \multicolumn{1}{c}{FP64} \\
\midrule
    \rowcolor{gray!30}
    $960$     & 4.9  & 5.1  & 3.6  & \textbf{3.3}   & 4.3 & 0.7 & 0.6 & \textbf{0.6}    & 0.7 & 1.3 \\
    $2560$    & 44.3 & 31.5 & \textbf{19.9} & 23.7   & 35.5 & 2.3 & 1.9 & \textbf{1.6}    & 2.3 & 6.9 \\
    \rowcolor{gray!30}
    $4096$    & 121.0 & 84.0 & \textbf{58.6} & 76.9  & 97.1 & 5.3 & 4.1 & \textbf{3.4}    & 5.2 & 17.0 \\
    $8192$    & 768.1 & 444.6 & \textbf{307.2} & 547.3  & 436.1 & 19.9 & 14.9 & \textbf{12.4}   & 19.4 & 69.6 \\
    \rowcolor{gray!30}
    $10000$   & 1373.0 & 779.2 & \textbf{529.5} & 973.8 & \textemdash & 29.5 & 22.0 & \textbf{18.3}   & 28.7 & \textemdash \\
        \bottomrule
    \end{tabular*}
\end{table}

We start by noticing that, for all grid sizes considered, the GADI solvers
converge to the desired accuracy even when most of the computation is
performed in lower precision. This behavior is consistent with the theory in
section~\ref{sec:ea}: as long as $A$ and the splitting matrices $H$ and $S$ are
sufficiently well-conditioned, mixed-precision GADI converges to high accuracy. 
GMRES-FP32 also successfully solves all problem sizes, but requires more 
runtime than GADI-FP32 for our largest problems. For the largest problem size, 
cuDSS-FP64 fails due to device memory overflow; sparse 
direct solvers like cuDSS typically require substantially more memory than
sparse iterative solvers.

The results of Table~\ref{tab:2d_convection_diffusion_performance} demonstrate that mixed-precision GADI achieves significant
performance gains over GADI-FP64 for this problem. GADI-BF16 is the fastest
GADI variant across all
reported sizes, delivering up to $2.59\times$ speedup over GADI-FP64 for
$n_g=10000$, and $1.42\times$ speedup over cuDSS-FP64 for the largest problem on which
cuDSS runs (i.e., $n_g = 8192$). Compared with GADI-FP32, GADI-BF16 also offers
significant speedup. For instance, for $n_g=10000$, GADI-BF16 is $1.47\times$
faster than GADI-FP32.

The two-dimensional steady-state convection-diffusion-reaction equation is a
problem for which we know that ADI approaches are some of the most efficient
solvers. This also translates to mixed precision when comparing GADI-FP32 with
GMRES-FP32, which is another iterative solver using an equivalent distribution
of precisions. Namely, we can observe that GADI-FP32 achieves significant
speedups over GMRES-FP32 on the two largest problems in 
Table~\ref{tab:2d_convection_diffusion_performance}, suggesting better
scalability properties for this problem.

Regarding memory, GADI-BF16 and
GADI-FP32 can reduce memory usage by storing the matrices $H$, $S$, and $\alpha I - N$ in lower precision, while still having to keep $A$ and the iterates $x_k$ in FP64 to
compute lines~\ref{line:alg1-res} and~\ref{line:alg1-update} of
Algorithm~\ref{alg1} in higher precision. Hence, we observe that
GADI-BF16 has a $1.61\times$ smaller memory usage than GADI-FP64 for
$n_g=10000$, and has a 
$5.61\times$ smaller memory usage than cuDSS-FP64 for $n_g=8192$.

\subsection{Three-Dimensional Convection-Diffusion Equation}
\label{subsec:3de}

\paragraph{Problem description.} 
We consider the three-dimensional convection-diffusion equation
\begin{equation*}
    -(u_{x_1 x_1} + u_{x_2 x_2} + u_{x_3 x_3}) + (u_{x_1} + u_{x_2} + u_{x_3}) = f(x_1, x_2, x_3),
\end{equation*}
defined on the unit cube \(\Omega = [0,1]^3\) with Dirichlet boundary conditions.
This model equation represents the balance between diffusion and convection
processes, commonly arising in transport phenomena, including fluid mechanics,
heat transfer, and mass transport \cite{morton1996numerical}.

The discretization via centered finite differences on a uniform grid with spacing 
\(h=1/(n_g+1)\) yields a sparse linear system \(Ax = b\), where the coefficient
matrix \(A \in \mathbb{R}^{n_g^3 \times n_g^3}\) has the Kronecker product structure
\begin{align*}
    A = T_x \otimes I \otimes I + I \otimes T_y \otimes I + I \otimes I \otimes T_z.
\end{align*}
Here, \(T_x, T_y, T_z \in \mathbb{R}^{n_g \times n_g}\) are tridiagonal 
matrices defined as \(T_x = \mathrm{Tridiag}(t_2, t_1, t_3)\) and
\(T_y = T_z = \mathrm{Tridiag}(t_2, 0, t_3)\), with \(t_1 = 2\), 
\(t_2 = -1 - r\), \(t_3 = -1 + r\), and \(r = 1/(2n_g + 2)\). The parameter
\(n_g\) denotes the number of grid points per spatial direction, resulting in a
linear system of dimension \(n=n_g^3\). The right-hand side vector 
\(b \in \mathbb{R}^{n_g^3}\) is constructed from the exact solution 
\(x = (1, 1, \ldots, 1)^T\).

Mixed precision GADI is applied with the symmetric/skew-symmetric splitting
\begin{align*}
H = \alpha I + \frac{A + A^T}{2}, \quad S = \alpha I + \frac{A - A^T}{2}.
\end{align*}

\paragraph{Performance comparison.}
Similarly to the two-dimensional convection-diffusion-reaction equation in
section~\ref{subsec:2de}, we summarize the runtimes and 
estimations of the memory peak consumption in 
Table~\ref{tab:3d_convection_diffusion_performance}. The largest linear system
has a dimension $n=1.3\times 10^8$.
All reported runs achieve a final relative residual accuracy of 
$\norm{b-A\widehat{x}_{k}}/\norm{r_0} \leq 10^{-6}$.
Note that for GADI-BF16, we did not use GPR for parameter selection as it 
proved ineffective for this problem; instead, we exceptionally hand-picked 
$\alpha$.

\begin{table}[t!]
    \centering
    \caption{
    \label{tab:3d_convection_diffusion_performance}
    Same as Table~\ref{tab:2d_convection_diffusion_performance} but for the 
    three-dimensional convection-diffusion equation.
    We report problems for $n_g \in
    \{180,\,256,\,320,\,360,\,400,\,450,\,512\}$, where $n=n_g^3$. All methods 
    achieve a relative residual $\norm{b-A\widehat{x}_{k}} / \norm{r_0} \leq
    10^{-6}$.
    }
\begin{tabular*}{0.99\linewidth}{c | rrr rr | rrr rr}
        \toprule
    \multirow{3}{*}{\shortstack{Grid\\Size\\$n_g$}} &
        \multicolumn{5}{c |}{\textbf{Runtime Performance (seconds)}} &
        \multicolumn{5}{c}{\textbf{Memory (GiB)}} \\
        \cmidrule(lr){2-6} \cmidrule(lr){7-11}
        & \multicolumn{3}{c}{\textbf{GADI}} &
        \multicolumn{1}{c}{\textbf{GMRES}} & \multicolumn{1}{c |}{\textbf{CUDSS}}
        & \multicolumn{3}{c}{\textbf{GADI}} &
        \multicolumn{1}{c}{\textbf{GMRES}} &
        \multicolumn{1}{c}{\textbf{CUDSS}} \\
        \cmidrule(lr){2-4} \cmidrule(lr){5-5} \cmidrule(lr){6-6}
        \cmidrule(l){7-9} \cmidrule(lr){10-10} \cmidrule(lr){11-11}
        & \multicolumn{1}{c}{FP64} & \multicolumn{1}{c}{FP32} &
        \multicolumn{1}{c}{BF16} & \multicolumn{1}{c}{FP32} &
        \multicolumn{1}{c |}{FP32} &
        \multicolumn{1}{c}{FP64} & \multicolumn{1}{c}{FP32} &
        \multicolumn{1}{c}{BF16} & \multicolumn{1}{c}{FP32} &
        \multicolumn{1}{c}{FP32} \\
\midrule
        \rowcolor{gray!30} 
        $180$     & 3.9          & \textbf{2.4}  & 2.8               & 3.7           & \textemdash & 3.0               & 2.4             & \textbf{2.0}    & 2.2           & \textemdash \\
        $256$     & 12.2          & \textbf{7.3}  & 8.9               & 15.8          & \textemdash & 7.7               & 5.9             & 5.5             & \textbf{5.4}  & \textemdash \\
        \rowcolor{gray!30}                                                                                                                                 
        $320$     & 25.7         & \textbf{15.3} & 18.4              & 41.5          & \textemdash & 14.7              & 11.2            & \textbf{9.4}    & 10.1          & \textemdash \\
        $360$     & 40.2         & \textbf{23.9} & 26.4              & 58.8          & \textemdash & 20.7              & 15.7            & \textbf{13.2}   & 16.4          & \textemdash \\
        \rowcolor{gray!30}                                                                                                                                 
        $400$     & 62.7         & \textbf{36.9} & 50.3              & 107.8         & \textemdash & 28.2              & 21.4            & \textbf{18.0}   & 22.4          & \textemdash \\
        $450$     & 101.1         & \textbf{60.6} & 137.3             & 199.9         & \textemdash & 40.0              & 30.3            & \textbf{25.4}   & 31.7          & \textemdash \\
        \rowcolor{gray!30}                                                                                                                                 
        $512$     & \textemdash   & \textemdash   & 390.7    & \textbf{328.3}         & \textemdash & \textemdash       & \textemdash     & \textbf{37.3}   & 46.5          & \textemdash \\
        \bottomrule
    \end{tabular*}
\end{table}

The results of Table~\ref{tab:3d_convection_diffusion_performance} reveal an
interesting performance trend. For moderate to large problem sizes up to $n_g=450$, 
GADI-FP32 consistently achieves the fastest runtime, delivering up to a 
$1.7\times$ speedup over GADI-FP64. Additionally, GADI-FP32 is also faster
than GADI-BF16 on most problem sizes, and is therefore (almost) always the most
time efficient GADI variants. Regarding the memory consumption, 
GADI-FP32 needs $1.32\times$ less memory compared with GADI-FP64 for $n_g=450$,
whereas GADI-BF16 achieves a greater reduction of $1.57\times$ over
GADI-FP64. For the largest problem $n_g=512$, both GADI-FP32 and GADI-FP64 exceed device memory, 
while GADI-BF16 remains within the 80GiB of GPU VRAM, thereby enabling problem 
sizes inaccessible to higher precision GADI variants.

Although the previous results might appear counterintuitive to a certain 
extent, they highlight an interesting case study.
Indeed, while a single inner-CG iteration in BF16 is computationally cheaper 
than one in FP32, GADI-BF16 requires a significantly larger number of outer 
iterations to converge. For instance, at $n_g=450$, GADI-FP32 converges in 88 
outer iterations with a cumulated number of 7493 inner CG iterations, whereas 
GADI-BF16 requires 955 outer iterations and 11673 cumulated inner CG iterations.
This significant increase in iterations for GADI-BF16 cancels the speedup 
benefit of individual BF16 inner-CG iterations.
Part of this underperformance is due to the failure of GPR to
identify a good parameter $\alpha$ for GADI-BF16.
Consequently, the manually selected $\alpha$ might not be optimal, or GADI-BF16 
might inherently require more iterations even for a well-selected $\alpha$.

Compared with the other solvers (i.e., cuDSS and mixed precision GMRES-based IR),
mixed precision GADI (almost) always offers the best performance. cuDSS-FP32 
cannot solve even the smallest grid listed due to memory exhaustion. GADI-FP32
is always significantly faster than GMRES-FP32 for $n_g \le 450$.
However, at the largest
grid size $n_g=512$, GADI-FP32 does not fit in GPU VRAM, while
GMRES-FP32 remains feasible (328.3s) and is faster than GADI-BF16 (390.7s).


\subsection{Complex Reaction-Diffusion Equation}
\label{subsec:complex_rd}

We consider the two-dimensional complex reaction-diffusion equation
\begin{equation}
    -\nu \Delta w + i V(x) w = f, \quad \text{in } \Omega = [0,1]^2,
\end{equation}
subject to homogeneous Dirichlet boundary conditions. Here, $\nu > 0$ is the 
diffusion coefficient, $i = \sqrt{-1}$ is the imaginary unit, and $V(x)$ is a
real-valued random potential function. This type of equation arises in various 
applications, including quantum mechanics and wave propagation in random media~\cite{pastur1992spectra,erdos2011universality}.

Discretizing the Laplacian using standard five-point central differences on a
uniform grid with $n_g$ points in each direction leads to a complex linear 
system $(L + iV)w = f$, where $L$ represents the discrete Laplacian scaled by 
$\nu$, and $V$ is a diagonal matrix containing the potential values at grid
points. To apply the GADI method, we rewrite this complex system in equivalent
real block form:
\begin{equation}
    \begin{bmatrix}
        L & -V \\
        V & L
    \end{bmatrix}
    \begin{bmatrix}
        w_r \\
        w_i
    \end{bmatrix}
    =
    \begin{bmatrix}
        f_r \\
        f_i
    \end{bmatrix},
\end{equation}
where $w = w_r + i w_i$ and $f = f_r + i f_i$. The resulting system matrix
$A \in \mathbb{R}^{n \times n}$ with $n = 2n_g^2$ is non-symmetric.

Mixed precision GADI is applied with the symmetric/skew-symmetric splitting
\begin{align*}
H = \alpha I + \frac{A + A^T}{2}, \quad S = \alpha I + \frac{A - A^T}{2}.
\end{align*}

In our experiments, we scale the random potential such that $V(x)=s\,\widetilde V(x)$
, where
$\widetilde V(x)$ is a random function with values uniformly distributed in $[0, 1]$, and
$s$ controls the magnitude of the reaction term $iV(x)w$. We set $s=10^4$ and
adjust the diffusion coefficient $\nu$ dynamically with the grid size $n_g$ to
maintain numerical difficulty, following the scaling $\nu = 10^{-5} \times (64/n_g)^2$.

\paragraph{Performance comparison.}
Similarly to the two previous sections~\ref{subsec:2de} and~\ref{subsec:3de},
we summarize the runtimes and estimations of the memory peak consumption for 
the complex reaction-diffusion problem in 
Table~\ref{tab:complex_rd_performance}. 
The largest linear system has a dimension $n=1.3\times 10^8$. All reported runs
achieve a final relative residual accuracy of 
$\norm{b-A\widehat{x}_{k}}/\norm{r_0} \leq 10^{-6}$.

\begin{table}[t!]
\centering
    \caption{
    \label{tab:complex_rd_performance}
    Runtime and memory performance for the complex reaction-diffusion equation.
    We report results for grid sizes $n_g \in \{1024, 2048, 4096, 5120, 8192\}$, with total system dimension $2n_g^2$.
    All methods achieve a relative residual $\norm{b-A\widehat{x}_{k}} / \norm{r_0} \leq 10^{-6}$.
    }
\begin{tabular*}{0.99\linewidth}{c | rrr rr | rrr rr}
        \toprule
    \multirow{3}{*}{\shortstack{Grid\\Size\\$n_g$}} &
        \multicolumn{5}{c |}{\textbf{Runtime Performance (seconds)}} &
        \multicolumn{5}{c}{\textbf{Memory (GiB)}} \\
        \cmidrule(lr){2-6} \cmidrule(lr){7-11}
        & \multicolumn{3}{c}{\textbf{GADI}} &
        \multicolumn{1}{c}{\textbf{GMRES}} & \multicolumn{1}{c |}{\textbf{CUDSS}}
        & \multicolumn{3}{c}{\textbf{GADI}} &
        \multicolumn{1}{c}{\textbf{GMRES}} &
        \multicolumn{1}{c}{\textbf{CUDSS}} \\
        \cmidrule(lr){2-4} \cmidrule(lr){5-5} \cmidrule(lr){6-6}
        \cmidrule(lr){7-9} \cmidrule(lr){10-10} \cmidrule(lr){11-11}
        & \multicolumn{1}{c}{FP64} & \multicolumn{1}{c}{FP32} &
        \multicolumn{1}{c}{BF16} & \multicolumn{1}{c}{FP32} &
        \multicolumn{1}{c |}{FP32} &
        \multicolumn{1}{c}{FP64} & \multicolumn{1}{c}{FP32} &
        \multicolumn{1}{c}{BF16} & \multicolumn{1}{c}{FP32} &
        \multicolumn{1}{c}{FP32} \\
\midrule
    \rowcolor{gray!30}
    1024 & 6.2   & 4.0  & \textbf{2.4}   & 5.1   & 13.7   & 1.2  & 1.0  & \textbf{0.9}  & 1.1  & 2.5 \\
    2048 & 21.8  & 13.1 & \textbf{8.1}   & 17.5  & 56.9   & 3.5  & 2.8  & \textbf{2.4}  & 3.1  & 8.6 \\
    \rowcolor{gray!30}
    4096 & 79.6  & 46.5 & \textbf{27.7}  & 59.2  & 257.1  & 12.8 & 9.8  & \textbf{8.3}  & 11.2 & 35.5 \\
    5120 & 121.7 & 72.4 & \textbf{40.2}  & 90.4  & 402.9  & 19.8 & 15.1 & \textbf{12.7} & 17.2 & 54.9 \\
    \rowcolor{gray!30}
    8192 & 309.0 & 180.2 & \textbf{98.7} & 225.6 & \textemdash & 49.9 & 37.9 & \textbf{31.9} & 43.4 & \textemdash \\
\bottomrule
\end{tabular*}
\end{table}

GADI-BF16 consistently outperforms all other methods for all problem sizes in 
both speed and memory efficiency. Specifically, for the
largest problem size $n_g=8192$, GADI-BF16 computes the solution in less than 100 
seconds, therefore achieving a speedup of approximately $3.1\times$ over GADI-FP64 and
$1.82\times$ over GADI-FP32. In terms of memory, GADI-BF16 requires 
31.9~GiB for the largest problem, which is $1.56\times$ less than GADI-FP64 and
$1.19\times$ less than GADI-FP32.

The direct solver cuDSS-FP32 is significantly slower and more 
memory-intensive than mixed precision GADI and GMRES-based IR. Even at moderate 
grid sizes (e.g., $n_g=4096$), cuDSS-FP32 is nearly $10\times$ slower than GADI-BF16 and consumes 
over $4\times$ more memory. For the largest grid $n_g=8192$, cuDSS-FP32 fails to
solve the system due to memory exhaustion. When compared against mixed
precision GMRES-based IR, we also observe that the mixed precision GADI approach
offers better performance for this class of problem. Indeed, GADI-FP32
outperforms GMRES-FP32 for all problem sizes in execution time and memory
consumption.

\section{Conclusion}
\label{sec:con}


In this article, we have introduced a mixed precision scheme for the General 
Alternating-Direction Implicit (GADI) framework, designed to solve large sparse 
linear systems efficiently on modern hardware. By decoupling the precisions 
used for the computationally intensive inner subsystems ($u_s$), the solution 
update ($u$), and the residual computation ($u_r$), our approach leverages the 
speed and memory advantages of low-precision arithmetic (such as FP32 and 
BF16) while maintaining the convergence and accuracy of high-precision 
solvers.

To assess the relevance and efficiency of our approach,
we first proceeded to the rounding error analysis of mixed precision GADI. In
doing so, we 
proved that the limiting accuracies of
the forward and backward errors are prescribed by the precision $u$ and $u_r$.
These results echo those in~\cite{bai2015numerical}, with the addition that the
limiting accuracy of the forward error can be made independent of $\kappa(A)$
by setting $u_r$ in a sufficiently high precision. Most importantly, we
quantified the error reductions achieved at each iteration, and showed that
the convergence rates of the errors are governed by the spectral radius of the
GADI iteration matrix and the conditioning of 
the splitting matrices together with the low precision $u_s$.
These results provide a theoretical foundation for leveraging aggressive 
low precision for solving the inner subsystems without compromising the final 
solution accuracy.


Then,
to mitigate the sensitivity of mixed precision GADI to the choice of the regularization parameter
$\alpha$, we explained that we employ a systematic, data-driven parameter selection strategy.
Using Gaussian Process Regression (GPR) trained on inexpensive small-scale instances,
the method predicts near-optimal parameters for large-scale problems given a 
set of precisions $u$, $u_r$, and $u_s$, and thereby reduces the need for
costly manual tuning.

We finally carried out experiments on large-scale 2D/3D convection-diffusion and
reaction-diffusion models to validate the effectiveness of our approach.
We demonstrated that the mixed precision GADI variants achieved substantial 
improvements over full double precision GADI, NVIDIA cuDSS direct solver, and
GMRES-based iterative refinement. Specifically, by setting $u_s$ in low precision (Bfloat16 or FP32), 
we achieved speedups of $2.6\times$, $1.7\times$, and $3.1\times$ for the 
largest problem sizes (up to $1.3\times 10^{8}$ unknowns) of the 2D, 3D
convection-diffusion and complex 
reaction-diffusion equations, respectively. 


\section*{Acknowledgments}
The numerical experiments were performed on the High Performance Computing Platform of Xiangtan University. We gratefully acknowledge its computational resources and technical support.

\section*{Funding}

The first and third authors were supported by the National Key Research
and Development Program of China (Grant No. 2023YFB3001604). The second author
was supported by the National Natural Science Foundation of China (No. 12288201
and W2433016) and the Beijing Natural Science Foundation (No. IS25038).

\bibliographystyle{siam}
\bibliography{IMANUM-refs}

\newpage
\appendix
\end{document}